\documentclass{amsart}
\newcommand{\mt}[1]{\mathtt{#1}}

\newcommand{\mF}{\mathcal{F}}

\usepackage{amsmath}
\usepackage{amsthm}
\usepackage{graphicx}
\usepackage{color}
\usepackage{amsfonts}
\usepackage{amssymb}
\usepackage{mathrsfs}
\usepackage{amscd}
\usepackage{url}

\newcommand{\re}{\mathbb{R}}
\newcommand{\cpx}{\mathbb{C}}

\newcommand{\N}{\mathbb{N}}

\renewcommand{\P}{\mathbb{P}}

\newcommand{\lmd}{\lambda}

\def\af{\alpha}
\def\bt{\beta}
\def\gm{\gamma}
\def\rank{\mbox{rank}}

\newcommand{\sig}{\sigma}
\newcommand{\Sig}{\Sigma}

\newcommand{\reff}[1]{(\ref{#1})}
\newcommand{\pt}{\partial}
\newcommand{\prm}{\prime}

\newcommand{\mc}[1]{\mathcal{#1}}

\newcommand{\bdes}{\begin{description}}
\newcommand{\edes}{\end{description}}

\newcommand{\bal}{\begin{align}}
\newcommand{\eal}{\end{align}}

\newcommand{\bnum}{\begin{enumerate}}
\newcommand{\enum}{\end{enumerate}}

\newcommand{\bit}{\begin{itemize}}
\newcommand{\eit}{\end{itemize}}

\newcommand{\bea}{\begin{eqnarray}}
\newcommand{\eea}{\end{eqnarray}}
\newcommand{\be}{\begin{equation}}
\newcommand{\ee}{\end{equation}}

\newcommand{\baray}{\begin{array}}
\newcommand{\earay}{\end{array}}

\newcommand{\bsry}{\begin{subarray}}
\newcommand{\esry}{\end{subarray}}

\newcommand{\bca}{\begin{cases}}
\newcommand{\eca}{\end{cases}}

\newcommand{\bcen}{\begin{center}}
\newcommand{\ecen}{\end{center}}

\newcommand{\bbm}{\begin{bmatrix}}
\newcommand{\ebm}{\end{bmatrix}}

\newcommand{\bmx}{\begin{matrix}}
\newcommand{\emx}{\end{matrix}}

\newcommand{\bpm}{\begin{pmatrix}}
\newcommand{\epm}{\end{pmatrix}}

\newcommand{\btab}{\begin{tabular}}
\newcommand{\etab}{\end{tabular}}

\newtheorem{theorem}{Theorem}[section]
\newtheorem{pro}[theorem]{Proposition}
\newtheorem{prop}[theorem]{Proposition}

\newtheorem{lemma}[theorem]{Lemma}

\newtheorem{defi}[theorem]{Definition}
\theoremstyle{definition}

\newtheorem{exm}[theorem]{Example}
\newtheorem{alg}[theorem]{Algorithm}

\newtheorem{remark}[theorem]{Remark}
\newtheorem{cond}[theorem]{Condition}

\begin{document}

\title[Generating Polynomials and Symmetric Tensor Decompositions]
{Generating Polynomials and Symmetric Tensor Decompositions}

\author{Jiawang Nie}
\address{
Department of Mathematics,  University of California San Diego,  9500
Gilman Drive,  La Jolla,  California 92093,  USA.
} \email{njw@math.ucsd.edu}

\begin{abstract}
This paper studies symmetric tensor decompositions.
%
%
For symmetric tensors, there exist linear relations
of recursive patterns among their entries. Such a relation
can be represented by a polynomial, which is called a {\it generating polynomial}.
The homogenization of a generating polynomial
belongs to the apolar ideal of the tensor.
A symmetric tensor decomposition can be determined by
a set of generating polynomials,
which can be represented by a matrix.
We call it a {\it generating matrix}.
Generally, a symmetric tensor decomposition
can be determined by a generating matrix
satisfying certain conditions. We characterize the sets
of such generating matrices and investigate their properties
(e.g., the existence, dimensions, nondefectiveness).
Using these properties, we propose methods for
computing symmetric tensor decompositions.
Extensive examples are shown to demonstrate
the efficiency of proposed methods.
\end{abstract}



\maketitle

\section{Introduction}

Let $m>0$ be an integer. A tensor $\mF$,
of order $m$ and on a vector space $V$ over a field $\mathbb{F}$,
can be viewed as a multi-linear functional
\[
f: \, V \times \cdots \times V \to \mathbb{F}.
\]
(In the above, $V$ is repeated $m$ times.)
The tensor $\mF$ is symmetric if the multi-linear functional $f$ is {\it symmetric},
i.e., $f(z_1,\ldots, z_m) = f(z_{\sig_1}, \ldots, z_{\sig_m})$
for all $z_1,\ldots, z_m \in V$ and for all
all permutations $(z_{\sig_1}, \ldots, z_{\sig_m})$ of $(z_1,\ldots, z_m)$.
Once a basis of $V$ is chosen, $f$ can be
represented by a multi-indexed array which is invariant
under permutation of indices. We denote by $\mt{T}^m(V)$ (resp., $\mt{S}^m(V)$)
the space of all such tensors (resp., symmetric tensors) of order $m$ and on $V$.

This paper focuses on symmetric tensors on
$\mathbb{F} = \cpx$ (the complex field) and $V=\cpx^{n+1}$
(the space of complex vectors of dimension $n+1$).
For convenience, the canonical unit vector basis of $\cpx^{n+1}$ is used.
Then a symmetric tensor $\mF \in \mt{S}^m(\cpx^{n+1})$ can be viewed
as an array, indexed by integer tuples $(i_1, \ldots, i_m)$ such that
\[
\mc{F} = (\mc{F}_{i_1 \ldots i_m})_{ 0 \leq i_1, \ldots, i_m \leq n}
\]
and $\mc{F}_{i_1 \ldots i_m}$ is invariant under permutations
of $(i_1,\ldots, i_m)$. Note that $\mF$ is uniquely
determined by the set of its upper triangular entries:
\be \label{df:up:tri}
\mbox{uptri}(\mF):=
\{ \mc{F}_{i_1 \ldots i_m}\}_{ 0 \leq i_1 \leq \cdots \leq  i_m \leq n}.
\ee

%
%
%

For vectors $u_1, \ldots, u_m \in \cpx^{n+1}$,
their outer product $u_1 \otimes \cdots \otimes u_m$
is the tensor in $\mt{T}^m(\cpx^{n+1})$ such that
\[
(u_1 \otimes \cdots \otimes u_m)_{i_1,\ldots,i_m} = (u_1)_{i_1}
\cdots (u_m)_{i_m}.
\]
(For convenience, we index $u\in \cpx^{n+1}$ by $j=0,1,\ldots,n$.)
An outer product like $u_1 \otimes \cdots \otimes u_m$ is called a rank-1 tensor.
For every $\mc{F} \in \mt{T}^m(\cpx^{n+1})$, there exist rank-1 tensors
$\mc{F}_1,\ldots, \mc{F}_r \in \mt{T}^m(\cpx^{n+1})$ such that
\be \label{dcmp:rank-r}
\mc{F} = \mc{F}_1 + \cdots + \mc{F}_r.
\ee
The smallest such $r$, denoted by $\rank(\mF)$, is called the rank of $\mF$.
For a vector $u \in \cpx^{n+1}$, denote the symmetric tensor power
\[
u^{\otimes m} := u \otimes \cdots \otimes u. \qquad ( u \mbox{ is repeated $m$ times}.)
\]
Tensors like $u^{\otimes m}$ are called rank-1 symmetric tensors.
Every symmetric tensor
is a linear combination of rank-1 symmetric tensors.
The {\it symmetric rank} of $\mc{F}$, denoted by $\rank_S(\mc{F})$,
is defined to be the smallest $r$ such that
\be \label{df:symrank}
%
%
\mc{F} =(u_1)^{\otimes m} + \cdots + (u_r)^{\otimes m}.
\ee \noindent
Clearly, $\rank(\mc{F}) \leq \rank_S(\mc{F})$.
The equation \reff{df:symrank} is called a
{\it symmetric tensor decomposition} (STD).
For symmetric tensors,  {\it their symmetric ranks
are just called ranks} for convenience, throughout the paper.
If $r = \rank_S(\mF)$, $\mF$ is called a rank-$r$ tensor and
the equation in \reff{df:symrank} is called a {\it rank decomposition}.

It is typically hard to determine the rank
of a symmetric tensor. However, for a general one,
its rank is given by a formula (see \reff{AlxHrchFml})
due to Alexander and Hirschowitz \cite{AlxHirs95}.
STDs have wide applications in chemometrics, signal
processing and higher order statistics (cf.~\cite{Com00}).
We refer to Comon et al.~\cite{CGLM08},
Lim \cite{Lim13}, and Landsberg \cite{Land12}
for symmetric tensors and their applications.

\subsection{Apolarity and catalecticant matrices}
\label{ssc:apolar:cat}

Symmetric tensor decompositions are closely related to apolarity,
which we refer to Iarrobino and Kanev \cite{IarKan99}.
To describe it, we need to index symmetric tensors by monomials,
or equivalently, by monomial powers.
Let $\cpx[\tilde{x}]:=\cpx[x_0,x_1,\ldots,x_n]$ be the ring of polynomials in
$\tilde{x} := (x_0, x_1,\ldots, x_n)$, with complex coefficients.
For $\theta = (\theta_0, \theta_1,\ldots,\theta_n) \in \N^{n+1}$, denote
\[
|\theta| := \theta_0 + \theta_1 + \cdots + \theta_n, \quad
\tilde{x}^\theta := x_0^{\theta_0}   x_1^{\theta_1}  \cdots  x_n^{\theta_n}.
\]
For $\mc{F} \in \mt{S}^m(\cpx^{n+1})$, we can index
it by $\theta$, with $|\theta|=m$, such that
\be \label{idx:theta}
\mc{F}_{\theta} := \mc{F}_{i_1 \ldots i_m}
\quad \mbox{ whenever} \quad
x_{i_1} \cdots x_{i_m}  = \tilde{x}^\theta.
\ee
Because $\mF$ is symmetric, there is a one-to-one correspondence
between $\{\mF_{\theta} \}_{|\theta|=m}$ and
$\mbox{uptri}(\mF)$ as in \reff{df:up:tri}.
The tensor $\mc{F}$ determines the polynomial in $\tilde{x}$:
\be \label{df:F(tdx)}
\mF(\tilde{x}) := \sum_{ 0 \leq i_1, \ldots, i_m \leq n}
\mF_{i_1 \ldots i_m}  x_{i_1} \cdots x_{i_m}.
\ee
Clearly, $\mF(\tilde{x})$ is a form
(i.e., a homogeneous polynomial) of degree $m$.
One can verify that (denote $\theta ! := \theta_0! \theta_1! \cdots \theta_n!$)
\be \label{F(tdx)=WarExpr}
\mF(\tilde{x}) = \sum_{ |\theta| = m}
  \mF_{\theta} \frac{m!}{ \theta ! }  \tilde{x}^\theta.
\ee

For a polynomial $p \in \cpx[\tilde{x}]$ with the expansion
\[
p(\tilde{x} ) = \sum_{ \theta = (\theta_0,  \theta_1, \ldots \theta_n)  }
p_{\theta} x_0^{\theta_0} x_1^{\theta_1} \cdots x_n^{\theta_n},
\]
define the operation $\circ$ between $p$ and a tensor $\mF$ as
\be \label{opr<poF>}
p \circ \mF := \sum_{ \theta = (\theta_0,  \theta_1, \ldots \theta_n)  }
p_{\theta}  \frac{\pt^{\theta_0 + \theta_1 + \cdots + \theta_n} }
{\pt x_0^{\theta_0} \pt x_1^{\theta_1} \cdots \pt x_n^{\theta_n} } \mF(\tilde{x}).
\ee
Note that $p \circ \mF$ is also a polynomial in $\tilde{x}$.
If $p$ is a form of degree $k\leq m$, then
$p \circ \mF$ is a form of degree $m-k$.
A polynomial $p \in \cpx[\tilde{x}]$ is said to be {\it apolar} to $\mF$ if
\be \label{ap:poF=0}
p \circ \mF  = 0,
\ee
that is, $p \circ \mF $ is the identically zero polynomial in $\tilde{x}$.
The {\it apolar ideal} of $\mF$ is the set
\be \label{df:apo:ideal}
\mbox{Ann}(\mF) :=\{ p \in \cpx[\tilde{x}]:
p \circ \mF  = 0 \}.
\ee
Indeed, $\mbox{Ann}(\mF)$ is a homogeneous ideal in $\cpx[\tilde{x}]$.
The operation $p \circ \mF $ is bilinear with respect to $p$ and $\mF$.
The {\it apolarity lemma} (cf.~\cite[Lemma~1.15]{IarKan99})
implies that a tensor $\mF$ has the decomposition
\[
\mc{F} =(u_1)^{\otimes m} + \cdots + (u_r)^{\otimes m},
\]
where $u_1,\ldots, u_r \in \cpx^{n+1}$ are pairwisely linearly independent,
if and only if the vanishing ideal of the points $u_1,\ldots, u_r$
(i.e., the set of polynomials that vanish on each $u_i$)
is contained in $\mbox{Ann}(\mF)$.
Thus, computing such a decomposition for $\mF$ is equivalent to
finding a set of forms $p_1, \ldots, p_N \in \mbox{Ann}(\mF)$,
which have finitely many and only simple zeros,
in the projective space $\P^n$ (see \S\ref{ssc:elm:AG}).

Let $ \cpx[\tilde{x}]_{k}^{hom}$ be the space of forms of degree $k$.
For $0 \leq k \leq m$, define the linear mapping
$C_{\mF}(m-k, k): \,  \cpx[\tilde{x}]_k^{hom} \to  \mt{S}^{m-k}(\cpx^{n+1})$,
\be \label{map:cat:mk}
p \mapsto  \frac{(m-k)!}{m!} \mbox{tensor}( p \circ \mF ).
\ee
where $\mbox{tensor}( p \circ \mF )$ denotes the tensor
$\mc{T} \in \mt{S}^{m-k}(\cpx^{n+1})$ such that
\[
\mc{T}( \tilde{x} ) \, = \, p \circ \mF.
\]
%
%
Under the canonical monomial basis of $\cpx[\tilde{x}]_k^{hom}$
and the canonical basis of $\mt{S}^{m-k}(\cpx^{n+1})$
corresponding to the monomial power indexing,
the representing matrix of the linear mapping $C_{\mF}(m-k, k)$
is the catalecticant matrix (cf.~\cite[\S1.1]{IarKan99})
\be \label{Cat:m-k:k}
\mbox{Cat}^{m-k,k}(\mF) :=
(\mF_{\vartheta + \theta})_{|\vartheta| = m-k, |\theta| = k  }.
\ee
Note that a form $q \in \cpx[\tilde{x}]_k^{hom}$ belongs to $\mbox{Ann}(\mF)$
if and only if the coefficient vector of $q$,
with respect to the basis of canonical monomials,
belongs to the null space of the catalecticant matrix $\mbox{Cat}^{m-k,k}(\mF)$.
The tensor $\mF$ has $m+1$ catalecticant matrices,
corresponding to $k=0,1,\ldots, m$. Among them, we are often interested in
the most square one, corresponding to $k = \lceil \frac{m}{2} \rceil$
(see \S\ref{sc:prlm} for the ceiling $\lceil \cdot \rceil$
and floor $\lfloor \cdot \rfloor$).
For convenience, we call it the {\it catalecticant matrix} of $\mF$ and denote
\be \label{CAT:most:sq}
\mbox{Cat}(\mF) := \mbox{Cat}^{ \lfloor \frac{m}{2} \rfloor,
\lceil \frac{m}{2} \rceil}(\mF).
\ee

\subsection{Existing work on STD}
\label{ssc:oldstd}

There exists much work on computing symmetric tensor decompositions,
by using apolarity, catalecticant matrices and other mathematical methods.
We refer to Comon et al. \cite{CGLM08},
Landsberg \cite{Land12} and the references therein
for recent results in the area.

For binary tensors (i.e., $n=1$), Sylvester's algorithm is often applied
to compute decompositions. It can be equivalently interpreted by apolarity as follows.
For a general $\mF \in \mt{S}^{m}(\cpx^2)$ with
$m= 2m_0$ (resp., $m=2m_0+1$), its symmetric rank equals $m_0+1$,
and a rank decomposition can be constructed from nontrivial zeros
of a general form $q$ belonging to the kernel of the mapping
$C_{\mF}(m_0-1,m_0+1)$ (resp., $C_{\mF}(m_0,m_0+1)$),
by the apolarity lemma. We refer to
Bernardi et al.~\cite{BerGimIda11},
Comas and Seiguer \cite{ComSei11} for binary tensor decompositions.

For higher dimensional tensors, the catalecticant method is often used.
Let $\{ \phi_1, \ldots, \phi_s \}$ be a basis of the kernel
of the mapping $C_{\mF}( \lfloor \frac{m}{2} \rfloor, \lceil \frac{m}{2} \rceil )$,
if it exists. If the homogeneous polynomial system
\be \label{phi(tx)=0}
\phi_1 (\tilde{x}) =  \cdots = \phi_s(\tilde{x}) = 0
\ee
has finitely many solutions in the projective space $\P^n$
and they are all simple,
then we can construct a decomposition for $\mF$
from the solutions, by the apolarity lemma.
If \reff{phi(tx)=0} has infinitely many solutions in $\P^n$,
or has a repeated solution, then we are not guaranteed to get an STD.
For tensors of generic ranks, the system \reff{phi(tx)=0}
typically has infinitely many solutions in $\P^n$,
and the catalecticant method may not be able to compute STDs.
We refer to Iarrobino and Kanev \cite[Chap.~4]{IarKan99}
and Oeding and Ottaviani~\cite[\S2.2]{OedOtt13}.

Brachat et al.~\cite{BCMT10} proposed a method for computing STDs,
by using properties of Hankel (and truncated Hankel) operators.
For $\mF \in \mt{S}^m(\cpx^{n+1})$ of rank $r$,
one can extend it to a higher order tensor
$\widetilde{\mF} \in \mt{S}^k(\cpx^{n+1})$
of order $k \geq m$ (depending on the value of $r$),
with new tensor entries as unknowns.
This method is equivalent to computing new entries of $\widetilde{\mF}$
such that the ideal, determined by the null space of
$\mbox{Cat}(\widetilde{\mF})$, is zero-dimensional and radical.
That is, determine whether there is an extension $\widetilde{\mF}$
such that the system \reff{phi(tx)=0}, corresponding to $\widetilde{\mF}$,
has finitely many and only simple solutions in $\P^n$.
The method starts with $r=1$; if such an extension $\widetilde{\mF}$
does not exist, then increase the value of $r$ by $1$, and repeat the process.
Theoretically, this method can compute STDs for all tensors.
In practice, it is usually very difficult to do that,
because checking existence of such an extension $\widetilde{\mF}$
is typically very hard.
%
%
We refer to Algorithm~5.1 of \cite{BCMT10}.

Symmetric tensor decompositions are equivalent to Waring decompositions
of homogeneous polynomials (cf.~\cite{Land12,OedOtt13}).
Oeding and Ottaviani~\cite{OedOtt13} proposed methods
for computing Waring decompositions.
They use Koszul flattening, tensor eigenvectors and vector bundles.
For even ordered tensors,
they have similar properties as the catalecticant method;
for odd ordered tensors,
they can compute STDs for broader classes of tensors.
Typically, these methods are efficient
when the ranks are lower than the generic ones.
For tensors of generic ranks, these methods may not be able to get STDs,
except some special cases (e.g., $4$-dimensional cubic tensors).
We refer to Algorithms~3, 4, 5 in \cite{OedOtt13}.

In addition to the above, there exists other work in the area.
We refer to Ballico and Bernardi \cite{BalBer12},
Bernardi et al. \cite{BerGimIda11},
Buczy\'{n}ska and Buczy\'{n}ski \cite{BucBuc10},
Comon et al. \cite{CGLM08},
Comon and Mourrain \cite{ComMou96},
Landsberg \cite{Land12} and the references therein.

\subsection{Generating polynomials} \,
\label{sbsc:rc}

As in \reff{idx:theta},
symmetric tensors in $\mt{S}^m(\cpx^{n+1})$ can be equivalently
indexed by monomials of degree equal to $m$ and in
$(x_0, x_1,\ldots, x_n)$.
Letting $x_0=1$, we can also equivalently index them
by monomials of degrees $\leq m$ and in $x := (x_1,\ldots, x_n)$.
Let $\cpx[x]:=\cpx[x_1,\ldots,x_n]$ be the ring of complex polynomials in $x$.
%
%
Let
\[
\N_m^n := \{ \af = (\af_1, \ldots, \af_n) \in \N^n :  |\af|   \leq m \}.
\]
For $\mc{F} \in \mt{S}^m(\cpx^{n+1})$, we can equivalently index
it by $\af \in \N_m^n$ such that
\[
\mc{F}_{\af} := \mc{F}_{m-|\af|, \af}.
\]
The right hand side indexing in the above is as in \reff{idx:theta}.
By setting $x_0=1$, the above indexing is equivalent to that
(denote $x^\af := x_1^{\af_1}\cdots x_n^{\af_n}$)
\be \label{index:F:af}
\mc{F}_{\af} = \mc{F}_{i_1 \ldots i_m}
\quad \mbox{if} \quad
x_{i_1}  \cdots x_{i_m} = x^\af.
\ee
Throughout the paper, we index tensors in $\mt{S}^m(\cpx^{n+1})$
by monomial powers $\af \in \N_m^n$, unless otherwise specified.

For symmetric tensors, there exist linear relations of recursive patterns
about their entries. Such linear relations can be represented by polynomials.
Denote by $\cpx[x]_m$ the space of polynomials
in $\cpx[x]$ with degrees $\leq m$. We define the bilinear product
$\langle \cdot, \cdot \rangle$ between
$p \in \cpx[x]_m$ and $ \mF \in \mt{S}^m(\cpx^{n+1})$ such that
%
%
\be \label{op:scrL:F}
 \langle p, \mc{F}  \rangle   =
\sum_{\af\in\N_m^n } p_\af \mc{F}_\af
\quad \mbox{ for }\quad
p =  \sum_{\af\in\N_m^n } p_\af x^\af.
\ee
In the above, $p_\af$ is the coefficient of $x^\af$ in $p$.

\begin{defi} \label{def:GR} \rm
We call $g\in \cpx[x]_m$ a {\it generating polynomial} (GP) for $\mc{F}$ if
\be \label{df:GRpq}
\langle g \cdot x^{\bt}, \mc{F} \rangle  = 0 \quad \forall\,
\bt  \in \N_{m-\deg(g)}^n.
\ee
\end{defi}

The condition \reff{df:GRpq} gives a set of equations
which are recursively generated by multiplying $g$
with monomials of appropriate degrees.
They are bilinear with respect to $g$ and $\mF$.
%
%
The name {\it generating polynomial} is motivated from the notion of
{\it recursively generated relation},
which was widely used for solving truncated moment problems
(cf.~Curto and Fialkow~\cite{CF96,CF98,CF05}).
Indexed by monomial powers in $\N_m^n$,
a tensor in $\mt{S}^m(\cpx^{n+1})$ can be viewed as a
truncated multi-sequence (tms).
A recursively generated relation for a tms can be represented by a polynomial
$g$ satisfying equations of patterns like \reff{df:GRpq} (cf.~\cite{CF98}).
This motivates the definition of generating polynomials for symmetric tensors.
Moreover, the word {\it generating} is corresponding to the fact that
generating polynomials can be used to determine
the entire tensor from a few of its entries.
This is demonstrated by Example~\ref{exmp:rk2}.

The concept of generating polynomials, as in Definition~\ref{def:GR},
is closely related to apolar ideals, although they look very differently.
As shown in Proposition~\ref{pro:apo<=>gen}, a polynomial $g$ is a GP
for $\mF$ if and only if
its homogenization $\tilde{g}(\tilde{x}) := x_0^{\deg(g)} g( x/x_0)$
is apolar to $\mF$, i.e., $\tilde{g}$
belongs to the apolar ideal $\mbox{Ann}(\mF)$ as in \reff{df:apo:ideal}.
%
%
For this reason, some readers may call generating polynomials
by alternative names like
``elements of an apolar ideal", ``apolar ideal elements", ``apolar forms".
Although the set of generating polynomials
is equivalent to the apolar ideal,
%
%
we find the formulation we give to be
more straightforward and easier to implement
in computation, because \reff{df:GRpq} gives
an explicit set of equations.

\begin{exm} \label{exmp:rk2}
Consider the tensor $\mc{F} \in \mt{S}^3(\cpx^3)$ whose slices
$\mF_{:,:,0}$, $\mF_{:,:,1}$, $\mF_{:,:,2}$ are respectively given as
\[
\left(
\baray{rrr}
     7  &  -3  &   9  \\
    -3  &  13  &  20  \\
     9  &  20  &  19  \\
\earay  \, \left| \,
\baray{rrr}
    -3 &   13  &  20  \\
    13 &  -27  &   6  \\
    20 &    6  &   6  \\
\earay  \, \right| \,
\baray{rrr}
     9  &  20  &  19  \\
    20  &   6  &   6  \\
    19  &   6  &  45  \\
\earay \right).
\]
We can check that the following
\[
\baray{rcl}
g_1 & =  &  14 -x_1 -4x_2 - 5 x_1^2,    \\
g_2 & =  & 4 -6 x_1 + 6 x_2 -5 x_1 x_2,  \\
g_3 & =  & 14 + 4x_1 + x_2 - 5 x_2^2
\earay
\]
are generating polynomials for $\mF$.
Their commons zeros are
\[
(-2 , -1), \quad (1,  2), \quad (2, -2).
\]
Using them, we can get a decomposition for $\mF$ as
\[
\mF = 3(1, -2 , -1)^{\otimes 3} +5 (1, 1,  2)^{\otimes 3}
- (1, 2, -2)^{\otimes 3}.
\]
The above is a rank decomposition for $\mF$, and $\rank_S(\mF)=3$.
(This can be implied by Lemma~\ref{lm:relF:ranks},
because the catalecticant matrix has rank $3$.)
\end{exm}

The meaning of the word {\it generating} can be illustrated
by the tensor in Example~\ref{exmp:rk2}.
It is important to observe that the entire tensor can be
determined by its first three entries $7, -3, 9 $
and its generating polynomials $g_1, g_2, g_3$.
The condition \reff{df:GRpq} implies that
(we use the monomial power indexing in \reff{index:F:af}){\smaller
\[
\mF_{20} = \frac{1}{5} ( 14 \mF_{00} - \mF_{10} - 4 \mF_{01} )  = 13, \quad
\mF_{11} = \frac{1}{5} ( 4 \mF_{00} - 6\mF_{10} +6 \mF_{01} )  =  20,
\]
\[
\mF_{02} = \frac{1}{5} ( 14 \mF_{00} +4 \mF_{10} + \mF_{01} )  = 19 , \quad
\mF_{30} = \frac{1}{5} ( 14 \mF_{10} - \mF_{20} - 4 \mF_{11} )  = -27,
\]
\[
\mF_{21} = \frac{1}{5} ( 14 \mF_{01} - \mF_{11} - 4 \mF_{02} )  = 6, \quad
\mF_{12} = \frac{1}{5} ( 4 \mF_{01} - 6\mF_{11} +6 \mF_{02} )  =  6 ,
\]
\[
\mF_{03} = \frac{1}{5} ( 14 \mF_{01} +4 \mF_{11} + \mF_{02} )  =  45.
\] \noindent}That is,
we can determine the entire tensor $\mF$, from the first
three entries and the polynomials $g_1,g_2, g_3$.
Indeed, this is also true for general tensors.
As we summarize in Prop.~\ref{pro:generate:F}, a general tensor of rank $r$ can
be determined by its first $r$ entries and a set of GPs.
This is another motivation for the concept of generating polynomials.

\subsection{Contributions}

This paper proposes methods for computing symmetric tensor decompositions,
by using generating polynomials.

As shown in Example~\ref{exmp:rk2}, an STD can be
constructed from the common zeros of a set of generating polynomials.
This motivates us to use GPs for computing STDs.
We propose a general approach for doing this
in a computationally efficient way. For general rank-$r$ tensors,
we construct a set of generating polynomials, say,
$\varphi_1, \ldots, \varphi_K$, which have $r$ common zeros.
To do this, we consider special generating polynomials,
which have computationally efficient formats (i.e.,
use low order monomials as few as possible),
and whose companion matrices are easily constructible.
Such a set of GPs can be represented by a matrix $G$,
which we call a {\it generating matrix} (GM).
For $\varphi_1, \ldots, \varphi_K$ to have $r$ common zeros,
their companion matrices are required to commute.
This gives a set of quadratic equations in the matrix $G$.
Such generating matrices are said to be {\it consistent}.

The paper studies properties of GPs for computing STDs.
We show that generally there is a one-to-one
correspondence between STDs and consistent GMs.
For general tensors of a given rank,
an STD uniquely determines a consistent GM,
and conversely, a nondefective GM uniquely determines an STD.
The nondefectiveness means that the generating polynomials
associated to the GM have no repeated zeros.
(This is often the case.) The cardinality of equivalent STDs
is equal to the cardinality of nondefective GMs.
The basic properties of GPs and GMs,
e.g., existence, dimensions and nondefectiveness,
are investigated in \S\ref{sc:GenRel}.

After studying such properties, we propose methods
for computing STDs. An STD can be obtained from a consistent GM,
which can be found by solving a set of quadratic equations.
There are two general types of mathematical methods for doing this.

\bit

\item  The first one is the type of {\it algebraic methods}.
They are based on solving polynomial systems
with classical algebraic methods (e.g., based on Gr\"{o}bner basis computations).
Their advantages include:
they are mathematically guaranteed to get STDs;
they can get all distinct tensor decompositions if there are finitely many ones;
they can compute the degrees of the fibers of decompositions
if there are infinitely many ones
(cf.~Examples~\ref{exmp:5.1}, \ref{exmp:5.2}).
A disadvantage is that algebraic type methods are limited to
small tensors, because of the typically high cost of algebraic methods.

\item The second one is the type of {\it numerical methods}.
A consistent GM is determined by a set of quadratic equations.
The classical numerical methods, for solving nonlinear systems
and nonlinear least-squares problems,
can be applied to compute consistent GMs. A major advantage is that
they can produce STDs for larger tensors,
because they avoid computations of Gr\"{o}bner and border bases.
A disadvantage is that they cannot be mathematically guaranteed
to produce an STD.
However, they work very efficiently in practice,
as demonstrated by numerical experiments.

\eit
The implementation of these methods
and their properties are presented in \S\ref{sc:symTD}.
They can be applied to all tensors of all ranks.
If they have generic ranks, i.e., given by the formula \reff{AlxHrchFml},
we can get rank decompositions.
This is an attractive property of our methods.
In the contrast, previously existing methods
may have difficulties for tensors of generic ranks.

We also give extensive numerical experiments for computing
STDs by using the proposed methods.
By algebraic methods, if there are infinitely many STDs,
we can get the degrees of fibers of decompositions.
This is shown in Examples~\ref{exmp:5.1} and \ref{exmp:5.2}.
If there are finitely many STDs, we can get all of them.
Please see examples in \S\ref{ssc:cmp:algbra}.
By numerical methods, we can get decompositions for much larger tensors.
The computational results are reported in \S\ref{ssc:comp:num}.

The proposed methods use some basic results from computational algebra.
For completeness of the paper,
we review such results in \S\ref{sc:prlm}.

\section{Preliminaries}
\label{sc:prlm}
\setcounter{equation}{0}

\noindent
{\bf Notation} \,
The symbol $\N$ (resp., $\re$, $\cpx$) denotes the set of
nonnegative integers (resp., real, complex numbers).
%
%
The cardinality of a finite set $S$ is denoted as $|S|$.
For a finite set $\mathbb{B} \subseteq \cpx[x]$ and a vector $v \in \cpx^n$, denote
\be \label{v:to:*m}
[v]_\mathbb{B}  := \big( p(v) \big)_{p \in \mathbb{B}},
\ee
the vector of monomials in $\mathbb{B}$ evaluated at the point $v$.
For $t\in \re$, $\lceil t\rceil$ (resp., $\lfloor t\rfloor$)
denotes the smallest integer not smaller
(resp., the largest integer not bigger) than $t$.
%
%
%
For a complex matrix $A$, $A^T$ denotes its transpose and $A^*$
denotes its conjugate transpose.
%
%
%
For a complex vector $u$, $\| u \|_2 = \sqrt{u^*u}$ denotes the standard Euclidean norm.
The $e_i$ denotes the standard $i$-th unit vector in $\N^n$.
%
For a tensor $\mc{F} \in \mt{T}^m(\cpx^{n+1})$,
its standard norm $\| \mc{F} \|$ is defined as
\be \label{tensor:norm}
\|\mF\|  = \Big(
\sum_{0 \leq i_1,\ldots, i_m \leq n}  |\mc{F}_{i_1 \ldots i_m}|^2
\Big)^{1/2}.
\ee
%
For two square matrices $X,Y$ of same dimension, denote their commutator:
\[
[X, Y] \, := \, XY-YX.
\]

\subsection{Elementary algebraic geometry}
\label{ssc:elm:AG}

In the space $\cpx^N$, we define the equivalence relation
$\sim$ as: for $a,b \in \cpx^N$, $a \sim b$
if and only if $a = \tau b$ for a complex number $\tau \ne 0$.
A set of all vectors that are equivalent to each other
is called an equivalence class.
The set of all nonzero equivalent classes in $\cpx^N$
is the projective space $\P^{N-1}$.
For basic concepts, such as {\it ideal}, {\it radicalness}, {\it projective variety},
{\it quasi-projective variety}, {\it affine variety},
{\it dimensions}, {\it codimensions}, {\it irreducibility}, {\it quotient space},
{\it open and closed sets in Zariski topology},
we refer to Cox, Little and O'Shea~\cite{CLO07},
Harris~\cite{Har} and Shafarevich~\cite{Sha:BAG1}.

On an affine or projective variety $V$, a {\it general} point for a property means that
the point belongs to a dense subset of $V$ on which the property holds,
in the standard Zariski topology.
A property is said to be {\it generically} true on $V$
if it holds on general points of $V$. For an irreducible variety $V$,
a property is generically true on $V$ if it is true
in a Zariski open subset of $V$.
We refer to Remark~2.1 of \cite{OedOtt13}.

\subsection{Basic properties of symmetric tensors}
\label{sbsc:symtsr}

The space $\mt{S}^m(\cpx^{n+1})$ of symmetric tensors
has dimension $\binom{n+m}{m}$. Its projectivization
is denoted as $\P \mt{S}^m(\cpx^{n+1})$, the set
of all nonzero equivalent classes in $\mt{S}^m(\cpx^{n+1})$.
The projective space $\P\mt{S}^m(\cpx^{n+1})$
has dimension $\binom{n+m}{m}-1$.
In $\P\mt{S}^m(\cpx^{n+1})$, let $\sig_r$ be the Zariski closure of
the set of equivalent classes of
$(u_1)^{\otimes m}+\cdots+(u_r)^{\otimes m}$,
with $u_1,\ldots,u_r \in \cpx^{n+1}$.
Equivalently, $\sig_r$ is the $r$-th secant variety of
the Veronese variety of degree $m$ and in $n+1$ variables.
The set $\sig_r$
is an irreducible variety in $\P\mt{S}^m(\cpx^{n+1})$
(cf.~\cite[\S5.1]{Land12}).

Determining the rank of a given tensor is a hard problem (cf.~\cite{HiLi13}).
However, for a general $\mc{F} \in \mt{S}^m(\cpx^{n+1})$ with $m>2$,
$\rank_S(\mF)$ is given by the formula {\small
\be \label{AlxHrchFml}
\overline{R}_S(m,n+1) :=  \left\lceil \frac{1}{n+1} \binom{n+m}{m} \right \rceil  +
\bca
1 & \text{ if } (m,n) \in \Omega, \\
0 & \text{ otherwise,}
\eca
\ee
\noindent}where
$\Omega =\{ (3,4), (4,2), (4,3), (4,4) \}$.
This is a result of Alexander and Hirschowitz \cite{AlxHirs95}.
For $r \leq \overline{R}_S(m,n+1)$,
the dimension of $\sig_r$ is (cf.~\cite[\S5.4]{Land12}){
\be \label{fml:dim:sig-r}
\dim \sig_r = \min \Big\{r(n+1)-1,
{\smaller \smaller \smaller \binom{n+m}{m}-1 } \Big \},
\ee  \noindent}except
the following cases:
\bit

\item if $m=2$ and $2\leq r\leq n$, then
$\dim \sig_r = \binom{r+1}{2} + r(n+1-r)-1$;

\item if $m=3,n=4,r=7$, then
$\dim \sig_r = \binom{n+m}{m} -2$;

\item if $m=4$, $2 \leq n \leq 4$, and $r=\binom{n+2}{2}-1$, then
$\dim \sig_r = \binom{n+m}{m} -2$.

\eit

In numerical computations, typically we cannot get exact
tensor decompositions because of round-off errors.
Usually we can only expect a decomposition
which is correct up to a small error.
Moreover, for $\mF \in \sig_r$, it is possible that $\rank_S(\mF)>r$,
because the set of rank-$r$ tensors might not be closed.
In practice, people are often
interested in the so-called {\it border rank} (cf.~\cite{Land12}).
The {\it symmetric border rank} of $\mF$,
denoted as $\rank_{SB}(\mc{F})$, is defined as
\be \label{df:bordrank}
\rank_{SB}(\mc{F}) = \min \left\{r: \, \mF \in \sig_r  \right\}.
\ee
The formula \reff{AlxHrchFml} gives
an upper bound for $\rank_{SB}(\mc{F})$.
%

Recall the catalecticant matrix (the most square one) as in
\reff{CAT:most:sq}. If we index $\mF$ as in \reff{index:F:af}, then
\be \label{df:CatMat:F}
\mbox{Cat}(\mF) \,= \, (\mF_{\af+\bt})_{|\af|\leq \lfloor \frac{m}{2} \rfloor,
|\bt|\leq \lceil \frac{m}{2} \rceil}.
\ee
%
The tensor ranks and catalecticant matrix ranks are related as follows.
\begin{lemma}  \label{lm:relF:ranks}
For all $\mF \in \mt{S}^m(\cpx^{n+1})$, it holds that
\be  \label{rk:Cat<=B<=S}
\rank\,\mbox{Cat}(\mF) \leq \rank_{SB}(\mF) \leq  \rank_S(\mF).
\ee
\end{lemma}
\begin{proof}
The second inequality in \reff{rk:Cat<=B<=S} follows from
\reff{df:bordrank}. We prove the first one.
For all $r$ with $\mF \in \sig_r$,
there exists a sequence of tuples $\{(u_1^k,\ldots, u_r^k)\}_{k=1}^\infty$ such that
\[
(u_1^k)^{\otimes m}+\cdots +(u_r^k)^{\otimes m}  \to \mF,
\quad \mbox{ as } \quad
k\to \infty.
\]
Each $\rank\,\mbox{Cat}\big( (u_i^k)^{\otimes m} \big) = 1$
if $u_i^k \ne 0$. For all $r$ with $\mF \in \sig_r$, {\small
\[
\mbox{Cat}\Big( \sum_{i=1}^r (u_i^k)^{\otimes m} \Big) =
\sum_{i=1}^r \mbox{Cat}\big( (u_i^k)^{\otimes m} \big),
\]
\[
\rank\, \mbox{Cat}\big(  \mF \big) \leq \lim_{k\to \infty}
\rank\, \mbox{Cat}\Big( \sum_{i=1}^r (u_i^k)^{\otimes m} \Big)
\leq r.
\]\noindent}So,
the first inequality in \reff{rk:Cat<=B<=S} is true.
\end{proof}

Lemma~\ref{lm:relF:ranks} can be used to determine
tensor ranks. If $\rank\,\mbox{Cat}(\mF)=r$
and $\mF$ has a decomposition of length-$r$,
then $\rank_S(\mF)=r$ by Lemma~\ref{lm:relF:ranks}.

As we have mentioned in \S\ref{sbsc:rc},
the concept of generating polynomials
is closely related to apolarity.
Recall the apolarity as in \reff{ap:poF=0} and the apolar ideal
$\mbox{Ann}(\mF)$ as in \reff{df:apo:ideal}.
Their relationship is stated in the following proposition.

\begin{prop} \label{pro:apo<=>gen}
Let $\mF \in \mt{S}^m(\cpx^{n+1})$ and $g\in \cpx[x]_k$ with $k\leq m$.
Then $g$ is a generating polynomial for $\mF$
if and only if its homogenization
$\tilde{g}(\tilde{x}) :=  x_0^k g(x/x_0)$
is apolar to $\mF$ (i.e., $\tilde{g} \in \mbox{Ann}(\mF)$).
\end{prop}
\begin{proof}
Write the polynomial $g$ in the standard expansion
\[
g(x) = \sum_{ \af =(\af_1, \ldots, \af_n)  \in \N_k^n }
g_\af \cdot  x_1^{\af_1}  x_2^{\af_2}  \cdots  x_n^{\af_n} .
\]
Its homogenization is
\[
\tilde{g}(\tilde{x}) = \sum_{ \af =(\af_1, \ldots, \af_n)  \in \N_k^n }
g_\af \cdot x_0^{k - |\af|} \cdot x_1^{\af_1}  x_2^{\af_2}  \cdots  x_n^{\af_n} .
\]
Consider the linear mapping $C_{\mF}(m-k,k)$ as in \reff{map:cat:mk}.
Under the canonical bases of $\cpx[x]_k^{hom}$ and $\mt{S}^{m-k}(\cpx^{n+1})$
with respect to monomial indexing, its representing matrix
is the catalecticant matrix $\mbox{Cat}^{m-k,k}(\mF)$ as in \reff{Cat:m-k:k},
whose rows and columns are respectively indexed by monomials in $\tilde{x}$
of degrees equal to $m-k$ and $k$ respectively.
By dehomogenization (i.e., let $x_0=1$),
the rows and columns of $\mbox{Cat}^{m-k,k}(\mF)$ can be equivalently
indexed by monomials in $x$
of degrees less than or equal to $m-k$ and $k$ respectively.
If we use the same indexing for $\mF$, then
\[
\mbox{Cat}^{m-k,k}(\mF) = (\mF_{\af+\bt})_{|\bt| \leq m-k, |\af| \leq k}.
\]
The condition \reff{df:GRpq} is equivalent to the equation
\[
\mbox{Cat}^{m-k,k}(\mF) \cdot vec(g) = 0,
\]
where $vec(g) :=(g_\af)_{|\af| \leq k}$ is the coefficient vector of $g$.
The above equation means that the form $\tilde{g}$
belongs to the kernel of $C_{\mF}(m-k,k)$,
that is, $\tilde{g} \circ \mF = 0$.
Therefore, $g$ is a generating polynomial if and only if
$\tilde{g}$ is apolar to $\mF$.
\end{proof}

\subsection{The fiber of decompositions}
\label{sbsc:fiber(F)}

Let $\sig_r$ be as in \S\ref{sbsc:symtsr}.
Denote by $\P(\cpx^{n+1})^r$ the projectivization of
the vector space $(\cpx^{n+1})^r$.
We define the fiber of length-$r$ decompositions of $\mF$ as
\be \label{df:fiber(F)}
\mbox{fiber}_r(\mF) := \Big \{ (u_1,\ldots,u_r) \in \P(\cpx^{n+1})^r  \mid
\mF = \Sig_{i=1}^r (u_i)^{\otimes m}
\Big\}.
\ee
(If $\rank_S(\mF)>r$, then $\mbox{fiber}_r(\mF)$ is empty.)
In fact, the fiber in \reff{df:fiber(F)}
is actually the fiber of the projection map from the incidence variety
to the embedded secant variety.
Each $(u_1,\ldots,u_r) \in \mbox{fiber}_r(\mF)$ is called
a {\it decomposing tuple} of $\mF$.
Clearly, if $(u_1,\ldots,u_r) \in \mbox{fiber}_r(\mF)$,
then every permutation of $(u_1,\ldots,u_r)$ also belongs to $\mbox{fiber}_r(\mF)$,
and $(\tau_1 u_1,\ldots, \tau_r u_r) \in \mbox{fiber}_r(\mF)$ for all
$\tau_1^m=\cdots=\tau_r^m=1$. Two decomposing tuples
$(u_1,\ldots, u_r), (\tilde{u}_1,\ldots, \tilde{u}_r)$ of $\mF$
are called {\it equivalent} if there exists
a permutation $(\nu_1,\ldots,\nu_r)$ of $(1,\ldots,r)$ and
unitary numbers $\tau_1^m=\cdots=\tau_r^m=1$ such that
$(\tilde{u}_1,\ldots,\tilde{u}_r) =
(\tau_1 u_{\nu_1}, \ldots, \tau_r u_{\nu_r})$
in the projective space $\P(\cpx^{n+1})^r$.
The set of all decomposing tuples, which are equivalent to each other,
is called a {\it decomposing class}.
The set of all decomposing classes in $\mbox{fiber}_r(\mF)$
is denoted as $\widetilde{\mbox{fiber}}_r(\mF)$.
So, $\widetilde{\mbox{fiber}}_r(\mF)$ represents the set of
all length-$r$ decompositions of $\mF$.
If the cardinality $L := |\widetilde{\mbox{fiber}}_r(\mF)|$ is finite,
then $\mF$ has $L$ distinct decompositions of length $r$.
If $L=1$, then $\mF$ has a {\it unique decomposition} of length $r$.

In computation, we often need to scale $u_i$
as $u_i = \tau (1, \theta_1, \ldots,\theta_n)$.
This usually requires $(u_i)_0 \ne 0$.
Denote the quasi-projective variety
\be \label{df:U0}
U_0 = \big \{ (u_1,\ldots,u_r) \in \P(\cpx^{n+1})^r:\,
(u_1)_0 \cdots (u_r)_0 \ne 0 \big \}.
\ee

\begin{prop} \label{pro:U0!=0}
Let $d:=r(n+1)-1 - \dim \sig_r$.
\bit
\item [(i)] If $d=0$, then $\mbox{fiber}_r(\mF) \subseteq U_0$
for a general $\mF \in \sig_r$.

\item [(ii)]  If $d>0$, then $\mbox{fiber}_r(\mF) \cap S \subseteq U_0$
for a general $\mF \in \sig_r$ and
for a general subspace $S \subseteq \P(\cpx^{n+1})^r$
of codimnension $d$.

\eit
\end{prop}
\begin{proof} Let
$
V = \big \{ (u_1,\ldots,u_r) \in \P(\cpx^{n+1})^r:\,
(u_1)_0  \cdots  (u_r)_0 = 0 \big \},
$
a hypersurface in $\P(\cpx^{n+1})^r$.
Consider the mapping:
\[
\rho:\, V \to \sig_r, \quad (u_1,\ldots,u_r) \mapsto
(u_1)^{\otimes m}+\cdots +(u_r)^{\otimes m}.
\]
We always have $\dim \rho(V) \leq r(n+1)-2$.

(i) If $d=0$, the closure $\overline{\rho(V)}$ is a subvariety of $\sig_r$,
with codimension $\geq 1$.
The complement $X:= \sig_r\backslash \overline{\rho(V)}$ is a
nonempty Zariski open subset of the irreducible variety $\sig_r$.
For all $\mF \in X$ and all $(u_1,\ldots, u_r) \in \mbox{fiber}_r(\mF)$,
we must have $(u_1,\ldots, u_r) \in U_0$.
This proves that $\mbox{fiber}_r(\mF) \subseteq  U_0$
for general $\mF \in \sig_r$.

(ii) Suppose $d>0$. If $\dim \overline{\rho(V)} < \dim \sig_r$,
then $X =\sig_r\backslash \overline{\rho(V)}$ is
a nonempty Zariski open subset of $\sig_r$.
For $\mF \in X$, we must have $\mbox{fiber}_r(\mF) \subseteq U_0$.
If $\dim \overline{\rho(V)} = \dim \sig_r$,
then $\overline{\rho(V)} =\sig_r$, because $\sig_r$ is irreducible.
For a general $\mF \in \sig_r$, the preimage $\rho^{-1}(\mF)$ has dimension $d-1$
(cf.~\cite[Theorem~7, \S6, Chapter~I]{Sha:BAG1}).
So, for a general subspace $S$ of $\P(\cpx^{n+1})^r$,
with codimnension $d$, the intersection $\rho^{-1}(\mF) \cap S$ is empty,
which means that $\mbox{fiber}_r(\mF) \cap S \subseteq U_0$.
\end{proof}

\subsection{Solving polynomial systems}
\label{sbsc:polyf=0}

Let $f_1,\ldots, f_k \in \cpx[x]$ and $I=\langle f_1, \ldots, f_k \rangle$
be the ideal generated by them. Consider the polynomial system
\be \label{f1k==0}
f_1(x) = \cdots = f_k(x) = 0.
\ee
Assume that $I$ is zero-dimensional, i.e., the quotient space
$\cpx[x]/I$ is finitely dimensional.
The set of all complex solutions to \reff{f1k==0} is
{\it the variety $\mc{V}(I)$} of $I$.
%
%
The number of its complex solutions (counting their multiplicities)
is equal to the dimension of $\cpx[x]/I$. Let $r:=\dim \cpx[x]/I$.
For each $x_i$, define the multiplication mapping:
\[
\mathscr{N}_{x_i}:\, \cpx[x]/I \to \cpx[x]/I, \quad
p \mapsto x_i p.
\]
Let $N_{x_i}$ be the representing matrix of $\mathscr{N}_{x_i}$,
under a basis $\{b_1,\ldots,b_r\}$.
It is called the {\it companion matrix} of $\mathscr{N}_{x_i}$.
The companion matrices $N_{x_1}, \ldots, N_{x_n}$ commute with each other,
so they share common eigenvectors. Stickelberger's Theorem
(cf.~Sturmfels~\cite[Theorem~2.6]{Stu02}) implies that
\be \label{Stkbg:Eig}
\mc{V}(I) = \big\{ (\lmd_1, \ldots, \lmd_n): \,   \exists v \in \cpx^r\backslash \{0\}, \,\,
   N_{x_i} v = \lmd_i v \mbox{ for each } i \big \}.
\ee
Moreover, $I$ is radical (i.e., every polynomial vanishing on $\mc{V}(I)$
belongs to $I$) if and only if the cardinality
$|\mc{V}(I)|$ equals the dimension of $\cpx[x]/I$.

Computing common eigenvectors is usually not convenient.
A practical method for computing $\mc{V}(I)$
is applying a generic linear combination of $N_{x_1}, \ldots, N_{x_n}$,
proposed by Corless, Gianni and Trager \cite{CGT97}.
Choose generic numbers $\xi_1>0,\ldots, \xi_n>0$
and scale them as $\xi_1+\cdots+\xi_n=1$. Let
\[
N(\xi) := \xi_1 N_{x_1} + \cdots + \xi_n N_{x_n}.
\]
Then, compute its Schur decomposition as
\be \label{s2:sur:Mxi*W}
Q^* N(\xi) Q = T :=
\bbm
T_{11} &  T_{12}  &   \cdots      &  T_{1s}  \\
       &  T_{22}  &         &   T_{2s}    \\
       &          & \ddots &   \vdots \\
       &          &        &  T_{ss}
\ebm,
\ee
where $Q \in \cpx^{r \times r}$ is unitary,
$T \in \cpx^{r \times r}$ is upper triangular,
the diagonal of each block $T_{jj}$ is a constant
(i.e., $T_{jj}$ has only one eigenvalue),
and different $T_{jj}$ has distinct diagonal entries. Let
$
\widetilde{N}_{x_i} = Q^* N_{x_i} Q
$
for each $i=1,\ldots,n$.
Then, we can partition $\widetilde{N}_{x_i}$
into a block matrix in the same pattern as $T$:
\[
\widetilde{N}_{x_i} =
\bbm
N^{(i)}_{11} &  N^{(i)}_{12}  &   \cdots      &  N^{(i)}_{1s}  \\
       &  N^{(i)}_{22}  &         &   N^{(i)}_{2s}    \\
       &          & \ddots &   \vdots \\
       &          &        &  N^{(i)}_{ss}
\ebm.
\]
(As shown in \cite{CGT97}, $\widetilde{N}_{x_i}$
is also a block upper triangular matrix with same block pattern as for $T$.)
For $j=1,\ldots,s$, let
\[
u_j := \Big(
\mbox{trace}(N^{(1)}_{jj}), \ldots, \mbox{trace}(N^{(n)}_{jj})
\Big)/\mbox{size}(T_{jj}).
\]
Then, the above $u_1,\ldots, u_s$ are the solutions to \reff{f1k==0},
and the size of $T_{jj}$ is the multiplicity of $u_j$.
We refer to \cite{CGT97} for the details.

To get the companion matrices $N_{x_1},\ldots,N_{x_n}$,
we need a basis for $\cpx[x]/I$.
%
%
In this paper, we mostly use the following basis
\be \label{monls:grlex}
\mathbb{B}_0 := \Big\{
\underbrace{
1, \, x_1, \, \ldots, \, x_n, \, x_1^2, \, x_1x_2, \, \ldots
}_{\mbox{ first $r$ monomials }}
\Big\},
\ee
the set of first $r$ monomials listed in the graded lexicographic order.

\subsection{Consistency of polynomial systems}

For $\mathbb{B}_0$ as in \reff{monls:grlex}, let
\be \label{mscrB12}
\mathbb{B}_1 := \big( \mathbb{B}_0 \cup x_1 \mathbb{B}_0 \cup \cdots \cup x_n \mathbb{B}_0)
\backslash \mathbb{B}_0.
\ee
The set $\mathbb{B}_1$ is called the border of $\mathbb{B}_0$
in the literature (cf.~\cite{BCMT10,LLMRT13}).
For convenience, by $\bt \in \mathbb{B}_0$ (resp., $\af \in \mathbb{B}_1)$
we mean that $x^\bt \in \mathbb{B}_0$ (resp., $x^\af \in \mathbb{B}_1)$.
Let $\cpx^{ \mathbb{B}_0 \times \mathbb{B}_1 }$ be the space of all complex matrices
indexed by $ (\bt, \af)  \in \mathbb{B}_0 \times \mathbb{B}_1$.
For $\af \in \mathbb{B}_1$ and $G \in \cpx^{ \mathbb{B}_0 \times \mathbb{B}_1 }$,
denote the polynomials
\be \label{vphi:W-af}
\varphi[G, \af] :=   \sum_{ \bt  \in \mathbb{B}_0 }
G(\bt,\af)  x^\bt -  x^\af \,  \in \cpx[x].
\ee
The coefficients of $\varphi[G,\af]$ only depend on
$G(:,\af)$, the $\af$-th column of $G$. Consider the polynomial system:
\be \label{vphiaf(x)=0}
\varphi[G](x) := \big ( \varphi[G,\af](x) \big)_{\af \in \mathbb{B}_1} = 0.
\ee
For the monomial sets $\mathbb{B}_0, \mathbb{B}_1$ as above,
the set $\{ \varphi[G, \af] \}_{\af \in \mathbb{B}_1 }$
is a border basis of $\langle \varphi[G] \rangle$,
the ideal generated by the tuple $\varphi[G]$.
We refer to \cite{LLMRT13} for border bases and
applications in solving polynomial systems.

We present sufficient and necessary conditions on $G$ such that
\reff{vphiaf(x)=0} has $r$ complex solutions, counting multiplicities.
For each $x_i$, define the linear mapping $\mathscr{M}_{x_i}$:
\be \label{df:scrMxi}
\mathscr{M}_{x_i}: \, \cpx[x]/\langle \varphi[G] \rangle \to
\cpx[x]/\langle \varphi[G] \rangle, \qquad p \mapsto x_i p.
\ee
Let $e_i$ denote the $i$-th standard unit vector of $\N^n$.
Define the matrix $M_{x_i}(G) \in \cpx^{\mathbb{B}_0 \times \mathbb{B}_0}$ as follows
($\mu,\nu \in \mathbb{B}_0$):
\be \label{df:Mxi(W)}
M_{x_i}(G)_{\mu, \nu} =
\bca
1  &  \text{ if } x_i \cdot x^\nu \in \mathbb{B}_0, \, \mu = \nu + e_i, \\
0  &  \text{ if } x_i \cdot x^\nu \in \mathbb{B}_0, \, \mu \ne \nu + e_i, \\
G(\mu, \nu+e_i)  &  \text{ if } x_i \cdot x^\nu \in \mathbb{B}_1.
\eca
\ee
The matrix $M_{x_i}(G)$ is affine linear in $G$.

When $\varphi[G]$ has $r$ complex zeros (counting multiplicities),
we can show that $\mathbb{B}_0$ is a basis
of $\cpx[x]/\langle \varphi[G] \rangle$ and
$M_{x_i}(G)$ is the representing matrix for $\mathscr{M}_{x_i}$.
For such case, the mappings $\mathscr{M}_{x_i}$ commute, i.e.,
\be \label{MiMj=MjMi}
M_{x_i}(G)M_{x_j}(G) - M_{x_j}(G)M_{x_i}(G) = 0
\, \, (1 \leq i < j \leq n).
\ee
So, \reff{MiMj=MjMi} is a necessary condition
for \reff{vphiaf(x)=0} to have $r$ complex solutions.
Indeed, \reff{MiMj=MjMi} is also sufficient.

\begin{prop} \label{phi[W]<=>Mcmu}
Let $\mathbb{B}_0, \mathbb{B}_1$ be as in \reff{monls:grlex}-\reff{mscrB12}.
The polynomial system \reff{vphiaf(x)=0} has $r$ complex solutions
(counting multiplicities) if and only if
\reff{MiMj=MjMi} holds. Moreover, \reff{vphiaf(x)=0} has $r$
distinct complex solutions if and only if
$M_{x_1}(G)$,$\ldots$,$M_{x_n}(G)$ are simultaneously diagonalizable.
\end{prop}

A result similar to Proposition~\ref{phi[W]<=>Mcmu} is
Theorem~4.2 of Brachat~et al. \cite{BCMT10}, which studies
conditions guaranteeing that a linear functional $\Lambda$ on $\re[x]_m$
can be extened to a linear funcational $\tilde{\Lambda}$ on $\re[x]$,
such that the Hankel operator associated to $\tilde{\Lambda}$
has a given rank $r$.
%
%
The condition~\reff{MiMj=MjMi} is used a lot for
computing tensor decompositions
(see Algorithms~\ref{alg:GR:algebraic}, \ref{alg:rc:dcmp}).
We invite the readers to compare the conditions in this paper
guaranteeing the existence of a rank-$r$ tensor decomposition
to those in Sections~3 and 4 of Brachat~et al. \cite{BCMT10},
as well as those in Theorems~2.4, 3.5 and 5.4 of
Oeding and Ottaviani \cite{OedOtt13}.
For completeness of the paper, we give a straightforward proof
for Proposition~\ref{phi[W]<=>Mcmu} due to lack of a suitable reference.

Let $\prec_{glx}$ be the graded lexicographic ordering on monomials:
\[
1 \prec_{glx} x_1 \prec_{glx}  \cdots \prec_{glx} \,
x_n \prec_{glx} x_1^2 \prec_{glx} x_1x_2 \prec_{glx} \cdots.
\]
For two polynomials $f,g$, we say that $f$ and $g$ are equivalent
with respect to the set $\{ \varphi[G] \}$
(and write $f \equiv g$ mod $\{ \varphi[G] \}$),
if the remainder of $f-g$, divided by polynomials in the tuple $\varphi[G]$
under the ordering $\prec_{glx}$, is zero.
We refer to \cite{CLO07} for polynomial divisions.
There is an equivalent characterization for \reff{MiMj=MjMi}.

\begin{cond}\label{cond:phi:cmut}
For all $x^\gm \in \mathbb{B}_0$ and all $1 \leq i < j \leq n$,
$G$ satisfies:
\bit

\item [(i)] If  $x^{\gm+e_i} \in   \mathbb{B}_1$ and $x^{\gm+e_j} \in  \mathbb{B}_1$,
then
\[
x_i \varphi[G,\gm+e_j] \equiv x_j \varphi[G,\gm+e_i]
\quad \mbox{mod} \quad \{ \varphi[G] \}.
\]

\item [(ii)] If  $x^{\gm+e_i} \in \mathbb{B}_0$ and $x^{\gm+e_j} \in \mathbb{B}_1$,
then
\[
x_i \varphi[G,\gm+e_j] \equiv  \varphi[G, \gm+e_i+e_j]
\quad \mbox{mod} \quad \{ \varphi[G] \}.
\]

\eit
\end{cond}

\begin{lemma} \label{lm:cd=cmmu}
Condition~\ref{cond:phi:cmut} is equivalent to the equation \reff{MiMj=MjMi}.
\end{lemma}
\begin{proof}
For each $i$, define the linear mapping
$\psi_i:\, \mbox{span}\{\mathbb{B}_0\} \to  \mbox{span}\{\mathbb{B}_0\}$
whose representing matrix is $M_{x_i}(G)$, under the basis $\mathbb{B}_0$.
One can check that $\psi_i$ maps each $x^{\gm} \in \mathbb{B}_0$ to the polynomial
$
x^{\gm+e_i} + \varphi[G,\gm+e_i]
$
in $\mbox{span}\{\mathbb{B}_0\}$.
Condition~\ref{cond:phi:cmut} precisely requires that
the mappings $\psi_i$ commute,
which is equivalent to that the matrices $M_{x_i}(G)$
commute, i.e., \reff{MiMj=MjMi} holds.
\end{proof}

\begin{lemma} \label{lm:crosred=0}
Let $\mathbb{B}_0, \mathbb{B}_1$ be as in \reff{monls:grlex}-\reff{mscrB12}.
If Condition~\ref{cond:phi:cmut} holds, then
for all $x^\mu, x^\nu \in \mathbb{B}_1$ such that $x^\mu \prec_{glx} x^\nu$
and $x^\theta x^\mu = x^\tau x^\nu$, we have
\[
x^\theta \varphi[G,\mu] - x^\tau \varphi[G,\nu] \equiv 0
\quad \mbox{mod} \quad \{ \varphi[G] \}.
\]
\end{lemma}
\begin{proof}
Since $x^\mu \prec_{glx} x^\nu$, we must have $x^\theta \succ_{glx} x^\tau$.
So, $\deg(x^\theta) \geq \deg(x^\tau)$.
We prove the lemma by induction on $\deg(x^\tau)$.
In the following, we omit the writing ``mod $\{ \varphi[G] \}$"
for the equivalent relation $\equiv$.

\smallskip \noindent
{\it Base Step: $\mathit{\deg(x^\tau)=0}$} \,  Then
$x^\theta x^\mu = x^\nu$ and $\deg(\theta)>0$.
Since $x^\nu \in \mathbb{B}_1$,
$\nu = \nu^\prm+e_l$ for some $\nu^\prm \in \mathbb{B}_0$ and $l$.
By $x^\mu \in \mathbb{B}_1$, we know
$\deg(x^{\nu^\prm}) \leq \deg(x^\mu)$. Moreover, it holds that
$\deg(x^\theta) = 1$, otherwise we can get
\[
\deg(x^\nu) = 1 + \deg(x^{\nu^\prm}) \leq 1 + \deg(x^\mu)
< \deg(x^\theta) + \deg(x^\mu)= \deg(x^\nu),
\]
a contradiction. So, there exists $i$ such that
$
\theta = e_i, \, \nu = \mu + e_i.
$
Note that
\[
\nu = (\mu-e_l) + e_i + e_l, \quad
(\mu-e_l) + e_i \in \mathbb{B}_0.
\]
We also have $l \ne i$, because otherwise
it results in the contradiction $\mu \not\in \mathbb{B}_1$.
Since $(\mu-e_l) + e_i \geq 0$, the $l$-th entry of $(\mu-e_l)$ must be nonnegative.
So, $\mu-e_l \in \mathbb{B}_0$.
By item (ii) of Condition~\ref{cond:phi:cmut}, we can get
\[
\varphi[G,\nu] = \varphi[G,\mu+e_i] =
\]
\[
\varphi[G, (\mu-e_l)+e_i+e_l]
\equiv  x^{e_i} \varphi[G, (\mu-e_l) + e_l] =
x^{\theta} \varphi[G, \mu].
\]

\smallskip \noindent
{\it Induction Step: $\mathit{\deg(x^\tau)>0}$} \,
We can write $\tau = \tau^\prm + e_\ell$ for some $\ell$.
If $\theta-e_\ell \geq 0$,
then $x^{\theta-e_\ell} x^\mu = x^{\tau^\prm} x^\nu$
and $\deg(x^{\tau^\prm}) < \deg(x^\tau)$.
By the induction, we have
\[
x^\tau \varphi[G,\nu] = x^{e_\ell} x^{\tau^\prm} \varphi[G,\nu] \equiv
x^{e_\ell} x^{\theta-e_\ell} \varphi[G,\mu] =
x^{\theta} \varphi[G,\mu].
\]
If $\theta-e_\ell \not\geq 0$, from $\theta+\mu = \tau + \nu$,
we know $\mu - e_\ell \geq 0$. Note that
$\mu - e_\ell \in \mathbb{B}_0 \cup \mathbb{B}_1$,
$x^{\theta} x^{\mu-e_\ell } = x^{\tau^\prm} x^\nu$,
and $\deg(x^{\tau^\prm}) < \deg(x^\tau)$.
\bit
\item  Suppose $\mu - e_\ell \in \mathbb{B}_1$. By the induction, we can get
\[
x^\tau \varphi[G,\nu] = x^{e_\ell} x^{\tau^\prm} \varphi[G,\nu] \equiv
x^{e_\ell} x^{\theta} \varphi[G,\mu-e_\ell]=
\]
\[
x^{\theta} x^{e_\ell} \varphi[G,\mu-e_\ell] \equiv
x^{\theta} \varphi[G,\mu].
\]
In the last equality above, the item (ii) of
Condition~\ref{cond:phi:cmut} is applied.

\item Suppose $\mu - e_\ell \in \mathbb{B}_0$.
Write $\theta = e_{k_1} + \cdots + e_{k_t}$ with
$k_1 \leq \cdots \leq k_t$.
Because $x^{\theta}  x^{\mu-e_\ell} \succ_{glx} x^{\nu} \in \mathbb{B}_1$,
there exists $j$ such that
\[
\mu-e_\ell + e_{k_1} + \cdots + e_{k_{j-1}} \in \mathbb{B}_0, \quad
\mu-e_\ell + e_{k_1} + \cdots + e_{k_j} \in \mathbb{B}_1.
\]
Let $\theta_0 = e_{k_1} + \cdots + e_{k_{j-1} }$,
$\theta_1 = e_{k_1} + \cdots + e_{k_j }$.
Then, $\theta - \theta_0 \geq 0$ and
$\theta - \theta_1 \geq 0$.
Let $\theta_2 := \theta - \theta_1$.
Note that $x^{\mu} \in \mathbb{B}_1$ and
$x^{\mu+\theta_0} \in \mathbb{B}_1$.
By repeatedly applying item (ii) of Condition~\ref{cond:phi:cmut},
we have $x^{\theta_0}\varphi[G,\mu] = \varphi[G,\mu+\theta_0]$ and
\[
x^{\theta} \varphi[G,\mu] = x^{\theta-\theta_0} x^{\theta_0} \varphi[G,\mu]
\equiv x^{\theta - \theta_0}  \varphi[G,\mu+\theta_0].
\]
By item (i) of Condition~\ref{cond:phi:cmut}, we have
\[
 x^{\theta-\theta_0}  \varphi[G,\mu+\theta_0] =
x^{\theta-\theta_0 -e_{k_j} }  x^{e_{k_j}} \varphi[G,\mu-e_\ell+\theta_0 + e_\ell]  \equiv
\]
\[
 x^{\theta_2 } x^{e_\ell} \varphi[G,\mu-e_\ell+\theta_0 + e_{k_j} ]
=
\]
%
%
\[
x^{\theta_2} x^{e_\ell} \varphi[G,\mu-e_\ell + \theta_1] =
x^{e_\ell} x^{\theta_2}  \varphi[G,\mu-e_\ell + \theta_1].
\]
Since $x^{\theta_2}  x^{\mu-e_\ell + \theta_1} = x^{\tau-e_\ell} x^{\nu}$,
by the induction, we get
\[
x^{\theta} \varphi[G,\mu] \equiv x^{e_\ell} x^{\theta_2}  \varphi[G,\mu-e_\ell + \theta_1] \equiv
x^{e_\ell} x^{\tau-e_\ell}  \varphi[G,\nu] =  x^{\tau}  \varphi[G,\nu].
\]
\eit
\end{proof}

\begin{lemma} \label{pro:bs=cmt}
Let $\mathbb{B}_0$, $\mathbb{B}_1$ be as in \reff{monls:grlex}-\reff{mscrB12}
and $G \in \cpx^{ \mathbb{B}_0 \times \mathbb{B}_1 }$.
Then $\mathbb{B}_0$ is a basis of $\cpx[x]/\langle \varphi[G] \rangle$
if and only if the equation \reff{MiMj=MjMi} holds.
\end{lemma}

\begin{proof}
The ``only if" direction was observed earlier,
because the companion matrices $M_{x_1}(G)$,$\ldots$, $M_{x_n}(G)$ commute.
We now prove the ``if" direction.  Suppose \reff{MiMj=MjMi} is true.
By Lemma~ \ref{lm:cd=cmmu}, Condition~\ref{cond:phi:cmut} holds.
Then, by Lemma~\ref{lm:crosred=0},
for all $x^\mu, x^\nu \in \mathbb{B}_1$ with $x^\mu \prec_{glx} x^\nu$,
we have
\[
x^\theta \varphi[G,\mu] - x^\tau \varphi[G,\nu] \equiv 0
\quad \mbox{mod} \quad \{ \varphi[G] \}
\]
whenever $x^\theta x^\mu = x^\tau x^\nu$.
The leading term of $\varphi[G,\af]$ is the monomial $x^\af$.
By Buchberger's algorithm (cf.~\cite{CLO07}), one can show that the set
\[
\Phi:=\{ \varphi[G, \af]: \, \af \in \mathbb{B}_1 \}
\]
is a Gr\"{o}bner basis of the ideal $\langle \varphi[G] \rangle$,
under the ordering $\prec_{glx}$.
%
%
For each $p \in \mbox{span}\{ \mathbb{B}_0 \}$, the remainder of
dividing $p$ by $\Phi$, under $\prec_{glx}$, is $p$ itself.
This means that, for each $p \in \mbox{span}\{ \mathbb{B}_0 \}$,
$
p \in \langle \varphi[G] \rangle
$
if and only if $p=0$. So, the monomials in $\mathbb{B}_0$
are linearly independent in $\cpx[x]/\langle \varphi[G] \rangle$,
which is spanned by $\mathbb{B}_0$.
Hence, $\mathbb{B}_0$ is a basis of $\cpx[x]/\langle \varphi[G] \rangle$.
\end{proof}

\begin{proof}[Proof of Proposition~\ref{phi[W]<=>Mcmu}]
The ideal $\langle \varphi[G] \rangle$ is zero dimensional, because
$\cpx[x]/\langle \varphi[G] \rangle$
is spanned by the finite set $\mathbb{B}_0$.
By Lemma~\ref{pro:bs=cmt}, \reff{MiMj=MjMi} holds if and only if
$\mathbb{B}_0$ is a basis of $\cpx[x]/\langle \varphi[G] \rangle$.
For such case, the number of complex solutions of \reff{vphiaf(x)=0},
counting multiplicities, is $r$, the cardinality of $\mathbb{B}_0$
(cf.~Sturmfels~\cite[Prop.~2.1]{Stu02}).
Each $M_{x_i}(G)$ is the companion matrix for the linear mapping
$\mathscr{M}_{x_i}$ defined as in \reff{df:scrMxi}, under the basis $\mathbb{B}$.
The system \reff{vphiaf(x)=0} has $r$ distinct complex solutions if and only if
the ideal $\langle \varphi[G] \rangle$ is radical, which is then equivalent to that
$M_{x_1}(G)$,$\ldots$,$M_{x_n}(G)$ are simultaneously diagonalizable
(cf.~Sturmfels~\cite[Corollary~2.7]{Stu02}).
\end{proof}

\section{Properties of generating polynomials}
\label{sc:GenRel}
\setcounter{equation}{0}

Let $\mc{F} \in \mt{S}^m(\cpx^{n+1})$ be a tensor with the decomposition
\be  \label{dcpA:u^m}
\mc{F} = (u_1)^{\otimes m} + \cdots +  (u_r)^{\otimes m},
\ee
for $u_1,\ldots, u_r \in \cpx^{n+1}$.
We index each $u_i$ by $j=0,1,\ldots,n$.
If each $(u_i)_0 \ne 0$, let
\be \label{dhmg:ui:to:vi}
v_i = \big( (u_i)_1, \ldots, (u_i)_n \big)/ (u_i)_0, \quad
\lmd_i = \big( (u_i)_0 \big)^m.
\ee
Then, \reff{dcpA:u^m} can be reduced to the decomposition
%
%
\be \label{dcpA:lmd*v}
\mc{F} = \lmd_1 (1, v_1)^{\otimes m} + \cdots + \lmd_r (1, v_r)^{\otimes m}.
\ee
Clearly, we can get \reff{dcpA:u^m} from \reff{dcpA:lmd*v} by
letting $u_i = \sqrt[m]{\lmd_i}(1,v_i)$.
They are equivalent if each $(u_i)_0 \ne 0$,
which is generically true by Proposition~\ref{pro:U0!=0}.

%
%
For $v_1,\ldots,v_r$ in \reff{dcpA:lmd*v}, one can show that
all polynomials vanishing on them are generating polynomials for $\mF$.
Interestingly, the reverse is also generally true.
We show that if there exists a set of generating polynomials for $\mF$,
which have finitely many common simple zeros,
then \reff{dcpA:lmd*v} can be constructed from their common zeros.
This is demonstrated by Example~\ref{exmp:rk2}.

Let $\mathbb{B}_0,\mathbb{B}_1$ be the set of monomials, as in
\reff{monls:grlex}-\reff{mscrB12}.
For convenience, by writing $\af \in \mathbb{B}_1$
(resp., $\bt \in \mathbb{B}_0$) we mean that $x^\af \in \mathbb{B}_1$
(resp., $x^\bt \in \mathbb{B}_0$).
A symmetric tensor can be equivalently indexed by monomials
(or equivalently, by monomial powers, cf.~\S\ref{sbsc:rc}).
By Definition~\ref{def:GR},
for $G \in \cpx^{\mathbb{B}_0 \times \mathbb{B}_1 }$,
the polynomials $\varphi[G,\af] \in \cpx[x]$ ($\af \in \mathbb{B}_1$),
defined as in \reff{vphi:W-af},
are generating polynomials for $\mc{F}$ if and only if
(see \reff{op:scrL:F} for the notation $\langle \cdot, \cdot \rangle$)
\be   \label{<vphi:af:W,A>=0}
\big \langle \varphi[G,\af] \cdot x^{\gm}, \mc{F} \big \rangle = 0
\quad \forall \, \af \in \mathbb{B}_1, \,
\forall \, \gamma \in \N_{m-|\af|}^n.
\ee
It is a set of linear equations in $G$, for given $\mF$.
For $G \in \cpx^{\mathbb{B}_0 \times \mathbb{B}_1 }$, recall that
\be \label{nt:vphi[G]}
\varphi[G] \, := \, \big(\varphi[G,\af]: \, \af \in \mathbb{B}_1 \big).
\ee

\begin{defi}  \label{def:GRmat}
Let $\mathbb{B}_0,\mathbb{B}_1$ be as in \reff{monls:grlex}-\reff{mscrB12}
and $r=|\mathbb{B}_0|$.
We call $G \in \cpx^{\mathbb{B}_0 \times \mathbb{B}_1 }$
a generating matrix for $\mF$
if \reff{<vphi:af:W,A>=0} is satisfied.
The set of all generating matrices for $\mF$ is denoted as
\be \label{df:g(F)}
\mathscr{G}(\mF) = \left \{ G \in \cpx^{ \mathbb{B}_0 \times \mathbb{B}_1 }: \,
 \reff{<vphi:af:W,A>=0} \mbox{ is satisfied } \right \}.
\ee
\end{defi}

By the construction of $\mathbb{B}_1$ as in \reff{mscrB12},
its cardinality $|\mathbb{B}_1|$ is at most $rn$,
but we don't have a closed formula for it.
Recall the set $\sig_r$ in \S\ref{sbsc:symtsr} (the set of tensors in $\mc{S}^m(\cpx^{n+1})$
whose symmetric border ranks are at most $r$).
For general $\mF \in \sig_r$, the set $\mathscr{G}(\mF)$ is nonempty.
For special $\mF \in \sig_r$, $\mathscr{G}(\mF)$ might be empty.
If $\rank_S(\mF) > r$, $\mathscr{G}(\mF)$ is usually empty.
This can be implied by Theorem~\ref{thm:dstcV=>lmd}.

\subsection{Correspondence to tensor decompositions}

Generally, the tensor decomposition \reff{dcpA:lmd*v}
can be constructed from a generating matrix $G$.
We refer to \S\ref{sbsc:fiber(F)} for
fibers of tensor decompositions, equivalent tensor decompositions,
decomposing classes, and the sets
$\mbox{fiber}_r(\mF)$, $\widetilde{\mbox{fiber}}_r(\mF)$.

To compute the decomposition \reff{dcpA:u^m},
it is enough to compute \reff{dcpA:lmd*v}
for the vectors $v_1,\ldots,v_r$.
Define the {\it affine fiber} of length-$r$ decompositions of $\mF$ as
\be \label{fiber:F}
\mbox{afiber}_r(\mF) = \big\{ \mbox{perm}(v_1,\ldots, v_r) : \,
\reff{dcpA:lmd*v} \mbox{ is satisfied} \big\},
\ee
where $\mbox{perm}(v_1,\ldots, v_r)$ is the set of
all permutations of $(v_1,\ldots, v_r)$.
The set $\mbox{afiber}_r(\mF)$ is empty if $\rank_S(\mF) > r$.
The cardinality of $\mbox{afiber}_r(\mF)$
is the number of distinct decompositions
\reff{dcpA:lmd*v} for $\mF$. If all $(u_i)_0 \ne 0$
(this is generally true by Prop.~\ref{pro:U0!=0}), then
the decomposing class of $(u_1,\ldots,u_r) \in \mbox{fiber}_r(\mF)$
can be uniquely determined by $\mbox{perm}(v_1,\ldots, v_r)$,
as in \reff{dhmg:ui:to:vi}.
%
%
So, generally there exists a one-to-one correspondence
between $\widetilde{\mbox{fiber}}_r(\mF)$ and $\mbox{afiber}_r(\mF)$.
If $\widetilde{\mbox{fiber}}_r(\mF)$ is a finite set,
the number of decomposing classes of $\mF$
is generally equal to the cardinality of $\mbox{afiber}_r(\mF)$.
Therefore, finding one (resp., all)
decomposing class of $\mF$ is generally equivalent to
finding one (resp., all) element in $\mbox{afiber}_r(\mF)$.

First, we show that each
$\mbox{perm}(v_1,\ldots, v_r) \in \mbox{afiber}_r(\mF)$
can be determined by a $G \in \mathscr{G}(\mF)$.
%
%
For all $G$, the ideal $\langle \varphi[G] \rangle$ is zero-dimensional,
because $\cpx[x]/\langle \varphi[G] \rangle$
is spanned by the finite set $\mathbb{B}_0$.
If $\mathbb{B}_0$ is a basis, the polynomial system
\be \label{vphi[W](x)=0}
\varphi[G](x)=0
\ee
has $r$ complex solutions, counting multiplicities.
The ideal $\langle \varphi[G] \rangle$ is radical
if and only if the $r$ solutions are all distinct.
Let $D$ be the open set
($\det$ denotes the determinant of a square matrix)
\be \label{set:D}
D =\left\{(v_1, \ldots, v_r) \in (\cpx^n)^r : \,
\det \Big( [v_1]_{\mathbb{B}_0} \quad \cdots \quad  [v_r]_{\mathbb{B}_0}  \Big) \ne 0 \right\}.
\ee

\begin{theorem}  \label{thm:dstcV=>lmd}
For all $\mc{F} \in \mt{S}^m(\cpx^{n+1})$, we have the properties:
\bit

\item [(i)] If $G \in \mathscr{G}(\mF)$ and $v_1,\ldots,v_r$
are distinct zeros of $\varphi[G]$,
then $\mbox{perm}(v_1,\ldots,v_r) \in \mbox{afiber}_r(\mF)$.

\item [(ii)] If $\mbox{perm}(v_1,\ldots,v_r) \in \mbox{afiber}_r(\mF)$ and
$(v_1,\ldots,v_r) \in D$,
then there exists a unique $G \in \mathscr{G}(\mF)$ such that
$\varphi[G](v_1)=\cdots =\varphi[G](v_r) =0$.

\eit
\end{theorem}

\begin{proof}
(i) For each $\af \in \mathbb{B}_1$, let $\widetilde{\varphi}[G,\af](\tilde{x})$
be the homogenization of $\varphi[G,\af](x)$,
in $\tilde{x} = (x_0, x_1, \ldots, x_n)$.
By Proposition~\ref{pro:apo<=>gen}, each form $\widetilde{\varphi}[G,\af](\tilde{x})$
is apolar to $\mF$. Clearly, the points
\[
(1, v_1), \ldots, (1, v_r)
\]
are pairwisely lindearly independent common zeros of
$\widetilde{\varphi}[G,\af](\tilde{x})$ in $\P^n$.
By the apolarity lemma (cf.~\cite[Lemma~1.15]{IarKan99}, there exist
$\lmd_1, \ldots, \lmd_r \in \cpx$ satisfying \reff{dcpA:lmd*v}, i.e.,
$\mbox{perm}(v_1,\ldots,v_r) \in \mbox{afiber}_r(\mF)$.

(ii) Consider the linear equations in $G$:
\[
\sum_{\bt \in \mathbb{B}_0} G(\bt, \af) (v_i)^\bt =  (v_i)^\af
\quad (\af \in \mathbb{B}_1, \, i=1,\ldots, r).
\]
They are equivalent to the matrix equation
\[
\bpm  [v_1]_{\mathbb{B}_0} & \cdots & [v_r]_{\mathbb{B}_0}  \epm^{T} G =
\bpm  [v_1]_{\mathbb{B}_1} & \cdots & [v_r]_{\mathbb{B}_1}  \epm^{T}.
\]
When $(v_1,\ldots,v_r) \in D$, it uniquely determines the matrix $G$.
For such $G$, $v_1,\ldots,v_r$ are common zeros of
$\varphi[G]$. For all $\af \in \mathbb{B}_1$
and all $\gm \in \N_{m-|\af|}^n$, we have
\[
\Big\langle x^\gm \varphi[G,\af], \mc{F} \Big \rangle =
\sum_{i=1}^k \lmd_i
\Big\langle x^\gm \varphi[G,\af], (1,v_i)^{\otimes m} \Big \rangle =
\sum_{i=1}^k \lmd_i (v_i)^\gm  \varphi[G,\af](v_i) = 0.
\]
That is, $G$ is a generating matrix for $\mF$.
\end{proof}

\begin{remark} \label{rmk:td:k<r}
If $[v_1]_{\mathbb{B}_0},\ldots, [v_k]_{\mathbb{B}_0}$
are linearly independent, with $k<r$, and
\[
\mF = \lmd_1 (1,v_1)^{\otimes m} + \cdots + \lmd_k (1,v_k)^{\otimes m},
\]
then there also exists $G \in \mathscr{G}(\mF)$ such that $\varphi[G]$
has $r$ distinct zeros. This can also be implied by Theorem~\ref{thm:dstcV=>lmd}(ii),
because, for generically chosen $v_{k+1},\ldots,v_r$, we still have
$(v_1, \ldots, v_r) \in D$ and
$\mbox{perm}(v_1, \ldots, v_r) \in \mbox{afiber}_r(\mF)$.
However, there are infinitely many such $G$ for this case.
\end{remark}

%
%
In Theorem~\ref{thm:dstcV=>lmd}(ii), if $(v_1,\ldots, v_r) \not \in D$, then
the desired $G$ there may not exist.
For instance, consider $\mathbb{B}_0=\{1,x_1,x_2\}$, $n=2$, $r=3$ and
\[
v_1 = (1,1), \quad v_2 = (2, 2), \quad v_3 = (3,3).
\]
One can check that $(v_1, v_2, v_3) \not\in D$ and
there is no $G$ such that the above $v_1,v_2,v_3$
are the common zeros of $\varphi[G]$.

%
%

\bigskip

Second, we show that for each
$\mbox{perm}(v_1,\ldots,v_r) \in \mbox{afiber}_r(\mF)$
it generally holds that $(v_1,\ldots,v_r) \in D$.
Recall the set $\sig_r$ defined as in \S\ref{sbsc:symtsr}.
%
%

\begin{theorem} \label{thm:rcW:exist}
Let $D$ be as in \reff{set:D} and $d=r(n+1)-1-\dim\sig_r$.
Then, for general $\mF \in \sig_r$, we have
$\dim\big( \mbox{afiber}_r(\mF) \big) = d$ and the properties:

\bit
\item [(i)] If $d=0$, then $(v_1,\ldots,v_r) \in D$ for all
$\mbox{perm}(v_1,\ldots,v_r) \in \mbox{afiber}_r(\mF)$.

\item [(ii)] If $d>0$, then  $(v_1,\ldots,v_r) \in D$ for all
$\mbox{perm}(v_1,\ldots,v_r) \in \mbox{afiber}_r(\mF)$
and $(v_1,\ldots,v_r) \in S$, for all
general subspaces $S \subseteq (\cpx^n)^r$ of codimension $d$.

\eit
\end{theorem}
\begin{proof}
First, consider the mapping
\[
\psi: D \times \P^{r-1} \to \sig_r, \quad
\big((v_1, \ldots, v_r), \lmd \big) \mapsto
\mF = \Sig_{i=1}^r \, \lmd_i (1,v_i)^{\otimes m}
\]
where $\lmd=(\lmd_1,\ldots, \lmd_r) \in \P^{r-1}$.
The dimension of the domain $D \times \P^{r-1}$ is $r(n+1)-1$.
The closure of the image $\psi(D \times \P^{r-1})$ is $\sig_r$.
Hence, for a general $\mF \in \sig_r$, $\dim \psi^{-1}(\mF) = d$
(cf.~\cite[Theorem~7, \S6, Chap.I]{Sha:BAG1}).
Note that, for each $\big((v_1, \ldots, v_r), \lmd \big) \in \psi^{-1}(\mF)$,
$\lmd$ is uniquely determined by $(v_1, \ldots, v_r)$.
This implies that $\dim \big( \mbox{afiber}_r(\mF) \big) = d$
for general $\mF \in \sig_r$.

Next, let $D^c := (\cpx^n)^r\backslash D$
and consider the mapping
\[
\phi: D^c \times \P^{r-1} \to \sig_r, \quad
\big((v_1, \ldots, v_r), \lmd \big) \mapsto
\mF = \Sig_{i=1}^r  \, \lmd_i (1,v_i)^{\otimes m}.
\]
The domain $D^c \times \P^{r-1}$ of $\phi$ has dimension $r(n+1)-2$.

(i) If $d=0$, then $\dim (D^c \times \P^{r-1}) = \dim \sig_r -1$,
so $\dim \phi(D^c \times \P^{r-1}) \leq \dim \sig_r-1$.
So, $O:=\sig_r \backslash \overline{ \phi(D^c \times \P^{r-1})}$
is a nonempty Zariski open subset of $\sig_r$.
For all $\mF \in O$, we must have $(v_1,\ldots,v_r) \in D$ for all
$\mbox{perm}(v_1,\ldots,v_r) \in \mbox{afiber}_r(\mF)$.
%
%

(ii) Suppose $d>0$. If the closure $\overline{\phi(D^c \times \P^{r-1})}$
is a proper subvariety of $\sig_r$, then, for general $\mF \in \sig_r$,
$\mF \not\in \phi(D^c \times \P^{r-1})$ and the conclusion is clearly true.
If $\overline{\phi(D^c \times \P^{r-1})} = \sig_r$,
that is, $\phi(D^c \times \P^{r-1})$ is dense in $\sig_r$.
So, it contains a Zariski open subset $\mathscr{X}$ of $\sig_r$
(cf.~\cite[Theorem~6, \S5, Chap.I]{Sha:BAG1}).
For a general $\mF \in \mathscr{X}$, $\dim \phi^{-1}(\mF) = d-1$
(cf.~\cite[Theorem~7, \S6, Chap.I]{Sha:BAG1}).
The projection $T$ of $\phi^{-1}(\mF)$ into $(\cpx^n)^r$ has dimension $\leq d-1$.
Thus, if $S$ is a general subspace of codimension $d$ in $(\cpx^n)^r$,
then $S \cap T = \emptyset$, which proves the item (ii).
\end{proof}

Third, we show that the entire coordinates of a symmetric tensor
can be determined from its generating matrix and a few of its coordinates.

\begin{prop}  \label{pro:generate:F}
For all $\mc{F} \in \mt{S}^m(\cpx^{n+1})$,
if $G \in \mathscr{G}(\mF)$, then for all
$\af \in \mathbb{B}_1$ and for all $\gamma \in \N^n$
with $|\gamma|+ |\af| \leq m$, we have
\be \label{gr:G=>F:r}
\mF_{ \af + \gamma } = \sum_{ \bt \in \mathbb{B}_0 }
G(\bt, \af) \mF_{\bt + \gamma} .
\ee
\end{prop}
\begin{proof}
The equation \reff{gr:G=>F:r} is implied
by the definition of generating polynomials as in
\reff{<vphi:af:W,A>=0} and that $G$ is a generating matrix.
\end{proof}
In \reff{gr:G=>F:r}, if $\gamma=0$ is chosen,
then the entries $\mF_{\af}$ ($\af \in \mathbb{B}_1$)
can be obtained as linear combinations of
$\mF_{\bt}$ ($\bt \in \mathbb{B}_0$),
with the coefficients $G(\bt,\af)$.  If
$\gamma=e_i$ ($i=1,\ldots,n$) is chosen, then we can get
the entries $\mF_{\af+e_i}$ ($\af \in \mathbb{B}_1$)
by the same linear combinations of $\mF_{\bt+e_i}$.
By repeating this process, if bigger values of $|\gamma|$ is chosen,
we can obtain all the entries of $\mF$,
which can be finally expressed as linear combinations of
$\mF_{\bt}$ ($\bt \in \mathbb{B}_0$).
That is, the entire $\mF$ can be determined (or recovered)
from its first $r$ entries, by using the generating matrix $G$.
We have seen this phenomena for Example~\ref{exmp:rk2}.

\subsection{Consistent generating polynomials}

By Theorem~\ref{thm:dstcV=>lmd}, the decomposition \reff{dcpA:lmd*v}
can be obtained from a generating matrix $G$ such that
$\varphi[G]$ has $r$ distinct zeros.
We first investigate when $\varphi[G]$ has $r$
complex zeros (counting multiplicities).
By Proposition~\ref{phi[W]<=>Mcmu}, this is the case if and only if the matrices
$M_{x_1}(G), \ldots, M_{x_n}(G)$, defined as in \reff{df:Mxi(W)}, commute.
%
%
Recall the set $\sig_r$ as in \S\ref{sbsc:symtsr}.

\begin{defi}
Let $\mathbb{B}_0,\mathbb{B}_1$ be as in \reff{monls:grlex}-\reff{mscrB12}
and $r = |\mathbb{B}_0|$.
We call $G\in \cpx^{\mathbb{B}_0 \times \mathbb{B}_1 }$
a consistent generating matrix for $\mF$ if
$G \in \mathscr{G}(\mF)$ and $\varphi[G]$
has $r$ complex zeros (counting multiplicities).
For such $G$, $\varphi[G,\af] \,(\af \in \mathbb{B}_1)$
are called consistent generating polynomials for $\mF$.
\end{defi}

By Proposition~\ref{phi[W]<=>Mcmu},
$G$ is a consistent generating matrix for $\mF$
if and only if $G \in \mathscr{C} \cap \mathscr{G}(\mF)$
where the set $\mathscr{C}$ is given as
\be \label{df:cnst:C}
\mathscr{C}:=  \Big\{ G \in \cpx^{\mathbb{B}_0 \times \mathbb{B}_1 } : \,
[M_{x_i}(G), \, M_{x_j}(G)] = 0 \,
\forall \, i, j \Big\}.
\ee

\begin{lemma} \label{lm:dimC=rn}
For $\mathbb{B}_0,\mathbb{B}_1$ as in \reff{monls:grlex}-\reff{mscrB12},
we have $\dim \mathscr{C}= rn$.
\end{lemma}
\begin{proof}
Since $1 \in \mathbb{B}_0$,
$
\{x_1, \ldots, x_n\} \subseteq \mathbb{B}_0 \cup \mathbb{B}_1.
$
By \reff{mscrB12}, for each $i$, there exists an integer $p_i >0$
such that $x_i^{p_i} \in \mathbb{B}_1$. Let
$
\mathbb{B}_c =   \{x_1^{p_1}, \ldots, x_n^{p_n} \}.
$
Consider the polynomial system
\be \label{eqn:G:piei}
\varphi[G, p_i e_i](x) =
\sum_{\bt \in \mathbb{B}_0 } G(\bt, p_i e_i) x^{\bt} - x_i^{p_i} =
0 \, \quad (i=1,\ldots,n).
\ee
Each equation has a leading term like $x_i^{p_i}$.
For arbitrary values of $G(:,\af)$ ($\af \in \mathbb{B}_c$),
the polynomial system \reff{eqn:G:piei} is always solvable and zero-dimensional.
The $r$ common zeros $v_1, \ldots, v_r$ of $\varphi[G]$
are algebraic functions in $G$.
Since they are also solutions to \reff{eqn:G:piei},
$v_1, \ldots, v_r$ can also be thought of as algebraic functions
in the free variables
\[
G_c = ( G(:, \af) )_{ \af \in  \mathbb{B}_c } .
\]

For vectors $u_1, \ldots, u_r \in \cpx^n$, denote the $r\times r$ square matrix
\[
V(u_1, \ldots, u_r) =
 \bpm  [u_1]_{\mathbb{B}_0} & \cdots & [u_r]_{\mathbb{B}_0}  \epm^{T}.
%
\]
Then, $V(v_1, \ldots, v_r)$ is clearly an algebraic matrix function in $G_c$.
We show that it is nonsingular, as a matrix function.
It is enough to show that $\det V(v_1, \ldots, v_r) \ne 0$
for some special values of $G$,
which can be chosen as follows.
First, select $\hat{u}_1, \ldots, \hat{u}_r \in D$,
where $D$ is as in \reff{set:D},
that is, $\det V(\hat{u}_1, \ldots, \hat{u}_r ) \ne 0$.
Then there exists a matrix $\hat{G}$ such that
$\hat{u}_1, \ldots, \hat{u}_r$ are common zeros of $\varphi[\hat{G}]$.
Actually, such $\hat{G}$ can be determined
by the equations $\varphi[\hat{G}](\hat{u}_j)=0$ and is given as
\[
\hat{G} = \bpm  [\hat{u}_1]_{\mathbb{B}_0} & \cdots & [\hat{u}_r]_{\mathbb{B}_0}  \epm^{-T}
\bpm  [\hat{u}_1]_{\mathbb{B}_1} & \cdots & [\hat{u}_r]_{\mathbb{B}_1}  \epm^{T}.
\]
The evaluation of each $v_j$ at $\hat{G}$ is $\hat{u}_j$.
So, $\det V\big(v_1, \ldots, v_r \big) \ne 0$ at $\hat{G}$,
and hence $V(v_1, \ldots, v_r)$ is a nonsingular matrix function.

For each $\af \in \mathbb{B}_1 \backslash \mathbb{B}_c$,
the common zeros $v_1, \ldots, v_r$ satisfy the equation
\[
V(v_1 , \ldots, v_r) G(:, \af) =  \bbm  (v_1)^\af & \cdots & (v_r)^\af \ebm^T.
\]
Since $V(v_1 , \ldots, v_r)$ is nonsingular,
each $G(:,\af)$ ($\af \in \mathbb{B}_1 \backslash \mathbb{B}_c$)
is an algebraic function in the free variables of $G_c$.
This means that the entries of $G_c$ form a
maximum algebraically independent set of $\cpx(\mathscr{C})$,
the field of rational functions on $\mathscr{C}$.
The transcendence degree of $\cpx(\mathscr{C})$ is $rn$,
so $\dim \mathscr{C} = rn$ (cf.~\cite[\S6, Chap.I]{Sha:BAG1}).
\end{proof}

The dimension of the intersection $\mathscr{C} \cap \mathscr{G}(\mF)$
is given as follows.

\begin{prop}  \label{thm:dim:C:g(F)}
Let $d = r(n+1) -1 - \dim \sig_r$.
If $\mF$ is a general tensor in $\sig_r$,
then $\dim \big( \mathscr{G}(\mF) \cap \mathscr{C} \big) = d$.
%
If, in addition, $d>0$ and
$\mathscr{H} \subseteq \cpx^{\mathbb{B}_0 \times \mathbb{B}_1}$
is a general subspace of codimension $d$,
then $\mathscr{G}(\mF) \cap \mathscr{C} \cap \mathscr{H}$
is a finite set.
\end{prop}
\begin{proof}
Let $D$ be the set as in \reff{set:D}. Consider the regular mapping
\[
\psi: D \times \P^{r-1} \to \sig_r, \quad
\big((v_1, \ldots, v_r), \lmd \big) \mapsto  \mF =
\Sig_{i=1}^r \, \lmd_i (1,v_i)^{\otimes m}.
\]
Clearly, the closure $\overline{\psi(D \times \P^{r-1})}$ is $\sig_r$.
By Theorem~7 of \cite[\S6, Chap. I]{Sha:BAG1}, there exists a
Zariski open subset $\mathscr{Y}$ of $\sig_r$
such that for each $\mF \in \mathscr{Y}$
\[
\dim \, \psi^{-1}(\mF)=r(n+1)-1-\dim \sig_r =d.
\]
Each $\big((v_1,\ldots, v_r), \lmd \big) \in \psi^{-1}(\mF)$
uniquely determines a matrix $G \in \mathscr{G}(\mF) \cap \mathscr{C}$,
by Theorem~\ref{thm:dstcV=>lmd}(ii). Indeed, $G$
can be determined by such $(v_1,\ldots,v_r)$ as
\[
G = \bpm  [v_1]_{\mathbb{B}_0} & \cdots & [v_r]_{\mathbb{B}_0}  \epm^{-T}
\bpm  [v_1]_{\mathbb{B}_1} & \cdots & [v_r]_{\mathbb{B}_1}  \epm^{T}.
\]
This implies that for all $\mF \in \mathscr{Y}$,
\be \label{dimGFCY>=gap}
\dim \big( \mathscr{G}(\mF) \cap \mathscr{C} \big)  \geq d.
\ee

Next, consider the mapping
($\P\cpx^{\mathbb{B}_0}$ denotes the projectivization of $\cpx^{\mathbb{B}_0}$)
\[
f: \mathscr{C} \times \P\cpx^{\mathbb{B}_0} \to \sig_r, \quad
(G, F) \mapsto  \mc{F} := f(G, F),
\]
which is defined such that the image $\mF = f(G, F)$ is
the tensor determined by $F\in \P\cpx^{\mathbb{B}_0}$
and the generating polynomials $\varphi[G,\af]$, that is,
\[
\mF_{\gm} = (F)_{\gm} \quad(\forall \,\gm \in \mathbb{B}_0), \quad
\mF_{\af} = \sum_{ \bt \in \mathbb{B}_0}
G_{\bt,\af} \mF_\bt \quad ( \forall \, \af \in \mathbb{B}_1),
\]
\[
\mc{F}_{\af+\gm} = \sum_{ \bt \in \mathbb{B}_0}
G_{\bt, \af} \mc{F}_{\bt+\gm} \, \quad
(\forall \, \af \in \mathbb{B}_1, \, \forall \, \gm \in \N^n_{m-|\af|}).
\]
This mapping $f$ is regular on $\mathscr{C} \times \P\cpx^{\mathbb{B}_0}$.
The variety $\sig_r$ is irreducible, while
$\mathscr{C}$ is not necessarily. We decompose $\mathscr{C}$ as
\[
\mathscr{C} = \mathscr{C}_1 \cup \cdots \cup \mathscr{C}_\ell,
\]
with $\mathscr{C}_1, \ldots, \mathscr{C}_\ell$
all irreducible and all distinct. Then, by Lemma~\ref{lm:dimC=rn},
\[
\max\left \{\dim \mathscr{C}_1, \ldots,
 \dim  \mathscr{C}_\ell \right\} = \dim \mathscr{C}   = rn,
\]
\[
\sig_r =  \overline{ f(\mathscr{C} \times \P\cpx^{\mathbb{B}_0} ) } =
 \overline{ f(\mathscr{C}_1 \times \P\cpx^{\mathbb{B}_0} ) }  \cup \cdots \cup
 \overline{ f(\mathscr{C}_\ell \times \P\cpx^{\mathbb{B}_0} ) }.
\]
Since $\sig_r$ is irreducible, some of
$\overline{ f(\mathscr{C}_1 \times \P\cpx^{\mathbb{B}_0} ) } , \ldots,
\overline{ f(\mathscr{C}_\ell \times \P\cpx^{\mathbb{B}_0} ) } $ are equal to $\sig_r$,
and the others are properly contained in $\sig_r$.
So, we can assume
\[
\sig_r =  \overline{ f(\mathscr{C}_1 \times \P\cpx^{\mathbb{B}_0} ) }  = \cdots =
\overline{ f(\mathscr{C}_s \times \P\cpx^{\mathbb{B}_0} ) }
\supsetneqq \overline{ f(\mathscr{C}_j \times \P\cpx^{\mathbb{B}_0} ) },
\quad j = s+1,\ldots, \ell.
\]
By the irreducibility of $\sig_r$, we know
\[
\dim \sig_r > \dim \overline{ f(\mathscr{C}_j \times \P\cpx^{\mathbb{B}_0} )},
\quad j = s+1,\ldots, \ell.
\]
So, the set
\[
\mathscr{Z}_0 := \sig_r \backslash \Big(
 \overline{ f(\mathscr{C}_{s+1} \times \P\cpx^{\mathbb{B}_0} ) }  \cup \cdots \cup
 \overline{ f(\mathscr{C}_\ell \times \P\cpx^{\mathbb{B}_0} ) }
\Big)
\]
is an open subset of $\sig_r$, in the Zariski topology.

For each $i=1,\ldots,s$, let $f_i = f|_{\mathscr{C}_i}$,
the restriction of $f$ on $\mathscr{C}_i$.
Then $f_i$ is a regular mapping from the irreducible variety
$\mathscr{C}_i \times \P\cpx^{\mathbb{B}_0}$ to the irreducible variety $\sig_r$.
By Theorem~7 of \cite[\S6, Chap. I]{Sha:BAG1},
$\sig_r$ has an open subset $\mathscr{Z}_i$ such that
\[
\dim f_i^{-1}(\mF) = r -1 + \dim \mathscr{C}_i - \dim \sig_r,
\]
for all $\mF \in \mathscr{Z}_i$. Let
$
\mathscr{Z} =  \mathscr{Z}_0 \cap  \mathscr{Z}_1 \cap \cdots \cap \mathscr{Z}_s.
$
For each $\mF \in \mathscr{Z}$, we have
\[
f^{-1}(\mF) = f_1^{-1}(\mF) \cup \cdots \cup f_s^{-1}(\mF),
\]
\[
\dim f^{-1}(\mF) =\max_{1 \leq i \leq s} \dim  f_i^{-1}(\mF),
\]
\be \label{dim:inv(F)<=gap}
\dim f^{-1}(\mF)  =
r -1 + \max_{1\leq i \leq s} \dim \mathscr{C}_i - \dim \sig_r \leq d.
\ee
Let $\mathscr{U} :=  \mathscr{Y} \cap \mathscr{Z}$,
which is again a Zariski open subset of $\sig_r$.
Then, for each $\mF \in \mathscr{U}$,
$\dim \big( \mathscr{G}(\mF) \cap \mathscr{C} \big) = d$ follows
from \reff{dimGFCY>=gap} and \reff{dim:inv(F)<=gap}.

When $d>0$, the intersection
$\mathscr{G}(\mF)\cap \mathscr{C} \cap \mathscr{H}$ is finite,
if $\mathscr{H} \subseteq \cpx^{\mathbb{B}_0 \times \mathbb{B}_1}$
is a general subspace of codimension $d$.
This follows from the definition of dimension
(cf.~\cite[Lecture~11]{Har}).
\end{proof}

By Proposition~\ref{thm:dim:C:g(F)},
when $d =0$, the intersection $\mathscr{G}(\mF)\cap \mathscr{C}$
is a finite set; when $d > 0$,
$\mathscr{G}(\mF)\cap \mathscr{C}$ has dimension $d$.
This property will be used in \S\ref{sc:symTD}.

\subsection{Distinct zeros of generating polynomials}
\label{sbsc:G:nodfc}

This subsection discusses when $\varphi[G]$ has $r$ distinct zeros
for a generating matrix $G$. By Proposition~\ref{phi[W]<=>Mcmu},
this is the case if and only if the companion matrices
$M_{x_1}(G), \ldots, M_{x_n}(G)$ are simultaneously diagonalizable.
Recall the set $\sig_r$ as in \S\ref{sbsc:symtsr}.

\begin{defi}  \label{defective:G}
Let $G \in \cpx^{ \mathbb{B}_0\times \mathbb{B}_1 }$
be a consistent generating matrix for $\mF$.
We call $\varphi[G]$ {\it nondefective} if it has $r$ distinct zeros,
and call $\varphi[G]$ {\it defective} if otherwise.
\end{defi}

By Theorem~\ref{thm:dstcV=>lmd}, if a nondefective $\varphi[G]$ is found,
then we can construct a tensor decomposition.
One concerns how often this is the case.
%
%
Note that $\varphi[G]$ has a repeated zero
only if all $M_{x_1}(G), \ldots, M_{x_n}(G)$
have a repeated eigenvalue, with a common eigenvector.
This occurs only if their discriminants are all zeros, i.e.,
\be \label{dis(MxiG)=0}
\mbox{dis}\Big(M_{x_1}(G)\Big) = \cdots =
\mbox{dis}\Big(M_{x_n}(G)\Big) = 0.
\ee
(In the above, $\mbox{dis}(X)$ denotes
the discriminant of the characteristic polynomial of $X$.
It equals zero if and only if $X$ has a repeated eigenvalue.
Cf.~\cite[\S7.5]{Stu02}.) For $G \in \mathscr{C}$,
if one of \reff{dis(MxiG)=0} is violated,
then $\varphi[G]$ is nondefective.
%
%
Denote the variety
\be
\mathscr{E}   = \left\{ G \in \cpx^{ \mathbb{B}_0 \times \mathbb{B}_1 }: \,
G \, \mbox{ satisfies } \, \reff{dis(MxiG)=0}
\right\}.
\ee

\begin{pro} \label{thm:g(F):C:nodfc}
Let $\mF \in \sig_r$ and $d = r(n+1) -1 - \dim \sig_r$.

\bit

\item [(i)] When $d=0$, if $\mathscr{G}(\mF) \cap \mathscr{C}$
does not intersect $\mathscr{E}$, then, for
every $G \in \mathscr{G}(\mF) \cap \mathscr{C}$,
$\varphi[G]$ is nondefective.

\item [(ii)] When $d>0$, if
$\mathscr{H} \subseteq \cpx^{\mathbb{B}_0 \times \mathbb{B}_1}$
is a subspace of codimension $d$ such that
$\mathscr{G}(\mF) \cap \mathscr{C} \cap \mathscr{H}$
does not intersect $\mathscr{E}$, then, for
every $G \in \mathscr{G}(\mF) \cap \mathscr{C} \cap \mathscr{H}$,
$\varphi[G]$ is nondefective.

\eit

\end{pro}
\begin{proof}
For $G  \in \mathscr{G}(\mF) \cap \mathscr{C}$,
the tuple $\varphi[G]$ is defective only if $G \in \mathscr{E}$.
Thus, if $\mathscr{G}(\mF) \cap \mathscr{C}$ does not intersect $\mathscr{E}$
(resp., $\mathscr{G}(\mF) \cap \mathscr{C} \cap \mathscr{H}$
does not intersect $\mathscr{E}$),
then every $G \in \mathscr{G}(\mF) \cap \mathscr{C}$
(resp., every $G \in \mathscr{G}(\mF) \cap \mathscr{C} \cap \mathscr{H}$)
does not belong to $\mathscr{E}$, that is,
$\varphi[G]$ is nondefective,
by Definition~\ref{defective:G}.
\end{proof}

By Proposition~\ref{thm:dim:C:g(F)}, for a general $\mF \in \sig_r$,
we know that: when $d=0$, the intersection
$\mathscr{G}(\mF) \cap \mathscr{C}$ is a finite set;
when $d>0$, the intersection
$\mathscr{G}(\mF) \cap \mathscr{C} \cap \mathscr{H}$
is a finite set, if $\mathscr{H}$
is a general subspace of codimension $d$.
Note that $\mathscr{E}$ is a variety of positive codimension.
Therefore, the assumptions in Proposition~\ref{thm:g(F):C:nodfc}
are usually satisfiable.
%
%

\section{Symmetric tensor decompositions}
\label{sc:symTD}
\setcounter{equation}{0}

Given a tensor $\mc{F} \in \mt{S}^m(\cpx^{n+1})$,
we want to find a decomposition
\be  \label{s4:dcpA:u^m}
\mc{F} = (u_1)^{\otimes m} + \cdots +  (u_r)^{\otimes m},
\ee
for some vectors $u_1, \ldots, u_r \in \cpx^{n+1}$.
The smallest such $r$ is the symmetric rank of $\mF$.
If $\mc{F}$ is a general tensor, then its rank is given
by the formula \reff{AlxHrchFml}.
Otherwise, \reff{AlxHrchFml} only gives an upper bound for the symmetric border rank.
Throughout this section, we assume the integer $r \geq \rank_S(\mc{F})$ is given.
Generally, a default value for $r$ can be chosen as in \reff{AlxHrchFml}.

When all $(u_i)_0 \ne 0$, \reff{s4:dcpA:u^m} is equivalent to
\be \label{s4:dcpA:lmd*v}
\mc{F} = \lmd_1 (1,v_1)^{\otimes m} + \cdots + \lmd_r (1,v_r)^{\otimes m}.
\ee
By Proposition~\ref{pro:U0!=0},
$(u_i)_0 \ne 0$ is generally satisfied.
Recall that $G$ is a generating matrix for $\mF$ if and only if
\reff{<vphi:af:W,A>=0} is satisfied. The polynomial tuple
$\varphi[G]$ is defined as in \reff{nt:vphi[G]}.
To get \reff{s4:dcpA:lmd*v}, by Theorem~\ref{thm:dstcV=>lmd},
it is enough to find $G \in \mathscr{G}(\mF)$ (see \reff{df:g(F)}) such that
$\varphi[G]$ has $r$ distinct zeros. Then,
$v_1,\ldots,v_r$ can be chosen to be the distinct zeros of $\varphi[G]$,
and the coefficients $\lmd_1,\ldots, \lmd_r$
can be determined from \reff{s4:dcpA:lmd*v}.
We propose to compute \reff{s4:dcpA:lmd*v} in two major steps:

\bit

\item  Find a generating matrix $G \in \mathscr{G}(\mF)$
such that the polynomial tuple $\varphi[G]$ has
$r$ distinct zeros (i.e., $\varphi[G]$ is nondefective).
%
%

\item Compute the zeros $v_1,\ldots, v_r$ of $\varphi[G]$,
and then determine the coefficients $\lmd_1,\ldots, \lmd_r$
satisfying the equation \reff{s4:dcpA:lmd*v}.

\eit

Let $\mathbb{B}_0,\mathbb{B}_1$ be the set of monomials, as in
\reff{monls:grlex}-\reff{mscrB12}.
For convenience, by writing $\af \in \mathbb{B}_1$
(resp., $\bt \in \mathbb{B}_0$)
we mean that $x^\af \in \mathbb{B}_1$
(resp., $x^\bt \in \mathbb{B}_0$).
For each $\af \in \mathbb{B}_1$,
define $A[\mc{F},\af] \in \cpx^{ \N_{m-|\af|}^n \times \mathbb{B}_0 }$
and $b[\mc{F},\af] \in \cpx^{ \N_{m-|\af|}^n }$ as
\be \label{df:Ab[F,af]}
\left\{ \baray{rcl}
A[\mc{F},\af]_{\gm , \bt} &=& \mc{F}_{\bt+\gm}, \quad
\forall \,  (\gm, \bt) \in \N_{m-|\af|}^n \times \mathbb{B}_0, \\
b[\mc{F},\af]_{\gm} &=& \mc{F}_{\af+\gm}, \quad
\forall \,  \gm \in \N_{m-|\af|}^n.
\earay\right.
\ee
Clearly, all $A[\mc{F},\af]$ and $b[\mc{F},\af]$ are linear in $\mc{F}$.
One can verify that $G \in \mathscr{G}(\mF)$ if and only if
each $G(:,\af)$ ($\af \in \mathbb{B}_1$) satisfies the equation
\be \label{af:AW=b}
A[\mc{F},\af] \, G(:,\af) = b[\mc{F},\af].
\ee
If, for some $\af \in \mathbb{B}_1$, the linear system \reff{af:AW=b}
is inconsistent, then it is most likely that $\rank_S(\mF) > r$.
(For such case, we need to increase the value of $r$.)
Otherwise, the vector $G(:,\af)$ satisfying \reff{af:AW=b}
can be linearly parameterized as
\be \label{W=c+N*omg:af}
G(:,\af)= c_\af + N_\af \omega_\af,
\ee
for a vector $c_\af$, a matrix $N_\af$
and a parameter $\omega_\af$. If it is not empty,
the set $\mathscr{G}(\mF)$ can be linearly parameterized as
\be \label{fun:W(omg)}
G(\omega) := C + N(\omega),
\ee
for some $C \in \cpx^{\mathbb{B}_0 \times \mathbb{B}_1 }$
and $N(\omega)$ linear in the parameter
\[
\omega:=(\omega_\af: \af \in \mathbb{B}_1 ).
\]

Recall that $G=G(\omega)$ is a consistent generating matrix for $\mF$
if and only if $\varphi[G]$ has $r$ zeros, counting multiplicities.
By Proposition~\ref{phi[W]<=>Mcmu}, this is equivalent to
\be \label{cmu:Mij:C+N*omg}
\Big[ M_{x_i}\big(G(\omega)\big), M_{x_j}\big(G(\omega)\big) \Big] = 0 \,\,
(1 \leq i <j \leq n).
\ee
Each $M_{x_i}\big(G(\omega)\big)$ is linear in $\omega$,
so \reff{cmu:Mij:C+N*omg} is a set of quadratic equations.

By Theorem~\ref{thm:dstcV=>lmd},
the decomposition \reff{s4:dcpA:lmd*v} can be found
by computing $\omega$ such that $\varphi[G(\omega)]$
has $r$ distinct zeros.
By Proposition~\ref{phi[W]<=>Mcmu}, this is equivalent to that
$M_{x_1}\big(G(\omega)\big)$, $\ldots$, $M_{x_n}\big(G(\omega)\big)$
are simultaneously diagonalizable, i.e., there exist
a nonsingular matrix $V$ and diagonal matrices $D_1, \ldots, D_n$ such that
\be \label{inV*Mxi*V=Di}
V^{-1}M_{x_1}(G(\omega))V = D_1, \, \ldots,
V^{-1}M_{x_n}(G(\omega))V = D_n.
\ee
Compared with \reff{cmu:Mij:C+N*omg}, the nonlinear system
\reff{inV*Mxi*V=Di} has new matrix variables $V$, $D_1$, $\ldots$, $D_n$.
Solving \reff{inV*Mxi*V=Di} in $(\omega, V, D_1, \ldots, D_n)$
is much harder, because of the big number of new extra variables.

We propose to compute symmetric tensor decompositions
by solving \reff{cmu:Mij:C+N*omg}, instead of \reff{inV*Mxi*V=Di}.
There are two approaches for doing this.
The first one uses algebraic methods,
while the second one uses numerical methods.

\subsection{An algebraic algorithm}

Recall that $\sig_r$ is the set of symmetric tensors
whose symmetric border ranks $\leq r$, defined as in \S\ref{sbsc:symtsr}.
Its dimension is given by \reff{fml:dim:sig-r}.
Let $d= r(n+1)-1-\dim \sig_r$ be the dimension gap.
By Proposition~\ref{thm:dim:C:g(F)}, if $d=0$,
\reff{cmu:Mij:C+N*omg} has finitely many solutions for general $\mF \in \sig_r$.
When $d>0$, this is true if we add $d$ generic linear equations to \reff{cmu:Mij:C+N*omg}.
When a polynomial system has finitely many solutions,
we can find all of them by classical algebraic methods
(cf.~\cite{CLO07,Stu02}).
This leads to the following algorithm.

\begin{alg} \label{alg:GR:algebraic}
An algebraic method for symmetric tensor decompositions.
\bit

\item [{\bf Input:}]  A general tensor $\mF \in \sig_r$.

\item [{\bf Output:}] One or several tuples $(u_1,\ldots,u_r)$
satisfying \reff{s4:dcpA:u^m}.

%
%

\item [Step 0:] Parameterize $G$ as in \reff{fun:W(omg)}.
Let $d = r(n+1)-1-\dim \sig_r$.

\item [Step 1:]
If $d=0$, solve \reff{cmu:Mij:C+N*omg} by an algebraic method.
If $d>0$, choose general $a_1, \ldots, a_d \in \cpx^\ell$
($\ell$ is the length of $\omega$)
and $b_1,\ldots,b_d \in \cpx$, then apply an algebraic method to
solve \reff{cmu:Mij:C+N*omg} together with
\be \label{rand:aw=c}
a_1^T\omega-b_1 = \cdots = a_d^T\omega-b_d = 0.
\ee
Let the solutions be $w_1,\ldots,w_N$, and $k=1$.

\item [Step 2:] Apply the method in \S\ref{sbsc:polyf=0}
to compute the complex zeros of $\varphi[G(w_k)]$,
say, $v_1,\ldots, v_r$. Then, determine $\lmd_1,\ldots,\lmd_r$
from the equation \reff{s4:dcpA:lmd*v}.

\item [Step 3:] For $i=1,\ldots,r$,
let $u_i := \sqrt[m]{\lmd_i}(1,v_i)$,
and output the tuple $(u_1,\ldots,u_r)$.

\item [Step 4:] If $k<N$, let $k:=k+1$ and go to Step~2;
otherwise, stop.

\eit

\end{alg}

In Step~0, if $G$ cannot be parameterized in the form \reff{fun:W(omg)}
(i.e., the set $\mathscr{G}(\mF)$ is empty),
then it is most likely that $\rank_S(\mF) > r$.
For such case, we need to increase the value of $r$.

The main task of Algorithm~\ref{alg:GR:algebraic} is in Step~1,
for solving the polynomial system \reff{cmu:Mij:C+N*omg},
together with \reff{rand:aw=c} if $d>0$.
When algebraic methods are applied to solve it,
we can get all the complex solutions.
On the other hand, such methods usually need to compute Gr\"{o}bner bases,
so they are usually efficient for small tensors.
We refer to \cite{CLO07,Stu02} for solving polynomial systems.

\begin{prop} \label{thm:symTD:prop}
For a general tensor $\mF \in \sig_r$, we have the properties:
\bit
\item [(i)] If $d=0$,
Algorithm~\ref{alg:GR:algebraic} can find
all the decomposing classes of $\mF$, i.e., the set
$\widetilde{\mbox{fiber}}_r(\mF)$ defined in \S\ref{sbsc:fiber(F)}.

\item [(ii)] If $d>0$, Algorithm~\ref{alg:GR:algebraic} can find
a finite slice of $\widetilde{\mbox{fiber}}_r(\mF)$,
parameterized by $a_1,\ldots, a_d$ and $b_1,\ldots, b_d$ in \reff{rand:aw=c}.

\eit
\end{prop}
\begin{proof}
By Proposition~\ref{thm:dim:C:g(F)}, the polynomial system \reff{cmu:Mij:C+N*omg},
together with \reff{rand:aw=c} if $d>0$,
has finitely many solutions for general $\mF \in \sig_r$.
When we apply algebraic methods to solve it,
all the complex solutions can be found.

(i) If $d=0$, then $\mbox{fiber}_r(\mF) \subseteq U_0$ (see \reff{df:U0})
for general $\mF \in \sig_r$, by Proposition~\ref{pro:U0!=0}.
So, \reff{s4:dcpA:u^m} is equivalent to \reff{s4:dcpA:lmd*v}.
By Theorem~\ref{thm:rcW:exist}, we have
$(v_1,\ldots,v_r) \in D$ (defined as in \reff{set:D}),
for all $\mbox{perm}(v_1,\ldots,v_r) \in \mbox{afiber}_r(\mF)$.
Hence, by Theorem~\ref{thm:dstcV=>lmd},
every decomposing class of $\mF$ uniquely determines a generating matrix $G$,
which is a solution of \reff{cmu:Mij:C+N*omg}.
Therefore, when $d=0$, Algorithm~\ref{alg:GR:algebraic} is able to find
all the decomposing classes of $\mF$.

(ii) If $d>0$, for general $\mF\in \sig_r$, Theorem~\ref{thm:rcW:exist}
implies that the set
\[
K := \{ \omega: \,  \reff{cmu:Mij:C+N*omg} \mbox{ is satisfied  and }
\varphi[G(\omega)] \mbox{ has no repeated zeros} \}
\]
has dimension $d$. For generically chosen $a_i,b_i$, the solution set
of \reff{rand:aw=c} must intersect $K$.
By Theorem~\ref{thm:dstcV=>lmd}(i), for each $\omega \in K$,
$G(\omega)$ determines a decomposing class of $\mF$.
%
%
Hence, Algorithm~\ref{alg:GR:algebraic} is able to get
a finite slice of $\widetilde{\mbox{fiber}}_r(\mF)$,
parameterized by the generically chosen
$a_1,\ldots,a_d$ and $b_1,\ldots,b_d$.
\end{proof}

To illustrate what the polynomial system \reff{cmu:Mij:C+N*omg} looks like,
we consider general tensors in $\mt{S}^3(\cpx^3)$, i.e.,
$n=2$, $m=3$ and $r=4$. So,
\[
\mathbb{B}_0 = \{ 1, x_1, x_2, x_1^2 \}, \quad
\mathbb{B}_1 = \{ x_1x_2, x_2^2, x_1^3, x_1^2 x_2 \}.
\]
For $G \in \cpx^{\mathbb{B}_0 \times \mathbb{B}_1}$,
the generating polynomials are given as:
\[
\baray{rcl}
G_{(0,0),(1,1)} + G_{(1,0),(1,1)}  x_1 + G_{(0,1),(1,1)} x_2 + G_{(2,0),(1,1)} x_1^2 &-& x_1x_2, \\
G_{(0,0),(0,2)} + G_{(1,0),(0,2)}  x_1 + G_{(0,1),(0,2)} x_2 + G_{(2,0),(0,2)} x_1^2 &-& x_2^2, \\
G_{(0,0),(3,0)} + G_{(1,0),(3,0)}  x_1 + G_{(0,1),(3,0)} x_2 + G_{(2,0),(3,0)} x_1^2 &-& x_1^3, \\
G_{(0,0),(2,1)} + G_{(1,0),(2,1)}  x_1 + G_{(0,1),(2,1)} x_2 + G_{(2,0),(2,1)} x_1^2 &-& x_1^2x_2.
\earay
\]
The $A[\mF,\af]$ and $b[\mF,\af]$ ($\af \in \mathbb{B}_1$) are given as follows:
\[
A[\mF,(1,1)] = A[\mF,(0,2)]=
\bbm
\mF_{00}  & \mF_{10}  & \mF_{01}  & \mF_{20} \\
\mF_{10}  & \mF_{20}  & \mF_{11}  & \mF_{30} \\
\mF_{01}  & \mF_{11}  & \mF_{02}  & \mF_{21}
\ebm,
\]
\[
A[\mF,(3,0)] = A[\mF,(2,1)]=
\bbm
\mF_{00}  & \mF_{10}  & \mF_{01}  & \mF_{20}
\ebm,
\]
\[
b[\mF,(1,1)] = \bbm \mF_{11} \\ \mF_{21} \\ \mF_{12} \ebm, \quad
b[\mF,(0,2)] = \bbm \mF_{02} \\ \mF_{12} \\ \mF_{03} \ebm,
\]
\[
b[\mF,(3,0)] = \bbm \mF_{30} \ebm, \quad  b[\mF,(2,1)] =  \bbm \mF_{21} \ebm.
\]
In the above, $\mF$ is indexed as in \reff{index:F:af}.
The companion matrices are
\[
M_{x_1}(G) =  \bbm
0  &  0  & G_{(0,0),(1,1)}  &  G_{(0,0),(3,0)} \\
1  &  0  & G_{(1,0),(1,1)}  &  G_{(1,0),(3,0)} \\
0  &  0  & G_{(0,1),(1,1)}  &  G_{(0,1),(3,0)} \\
0  &  1  & G_{(2,0),(1,1)}  &  G_{(2,0),(3,0)} \\
\ebm,
\]
\[
M_{x_2}(G) =  \bbm
0  &   G_{(0,0),(1,1)}  &  G_{(0,0),(0,2)}  & G_{(0,0),(2,1)} \\
0  &   G_{(1,0),(1,1)}  &  G_{(1,0),(0,2)}  & G_{(1,0),(2,1)} \\
1  &   G_{(0,1),(1,1)}  &  G_{(0,1),(0,2)}  & G_{(0,1),(2,1)} \\
0  &   G_{(2,0),(1,1)}  &  G_{(2,0),(0,2)}  & G_{(2,0),(2,1)} \\
\ebm.
\]
In the parametrization \reff{fun:W(omg)},
the length of the unknown vector $\omega$ is $8$.
The dimension gap $d =2$.
After adding two general linear equations as in \reff{rand:aw=c},
we can reduce $\omega$ to a vector of $6$ unknowns.
Finally, we get a polynomial system of $12$ quadratic equations
and in $6$ unknowns. It can be solved efficiently by polynomial system solvers.
For general $\mF \in \mt{S}^3(\cpx^3)$, the resulting polynomial
system has $7$ distinct solutions,
so we can get $7$ rank decompositions.
%
%
This is demonstrated by Example~\ref{exmp:5.1}.

\subsection{A numerical algorithm}

In this subsection, we propose numerical methods
for solving \reff{cmu:Mij:C+N*omg},
together with \reff{rand:aw=c} if $d>0$.
There exist classical numerical methods for solving nonlinear systems
and nonlinear least-squares problems, e.g., Gauss-Newton,
trust region, and Levenberg-Marquardt methods.
We refer to \cite{DenSch83,Kel95,More78,yyx11} for the work in this area.
In practice, numerical methods are often more
efficient for solving large polynomial systems.

Most numerical methods need a good starting point.
For tensor decompositions, this can be done heuristically as follows.
First, solve the problem
\be \label{nLS:A=u^m}
\min_{u_1, \ldots, u_r \in \cpx^{n+1} }
\quad \| \mc{F} - \big( (u_1)^{\otimes m} +
\cdots + (u_r)^{\otimes m} \big)  \|^2,
\ee
with a random starting point. Let $(u_1^0, \ldots, u_r^0)$
be a computed solution of \reff{nLS:A=u^m}.
If $\mc{F} = (u_1^0)^{\otimes m} + \cdots + (u_r^0)^{\otimes m}$, then
\reff{s4:dcpA:u^m} is found and we are done. Otherwise,
write each $u_i^0 = \tau_i (1, v_i^0)$, with $v_i^0 \in \cpx^n$.
(If the first entry $(u_i^0)_0$ is zero or tiny, we can choose $v_i^0$ randomly.)
Let $G^0 \in \cpx^{\mathbb{B}_0 \times \mathbb{B}_1}$ be as
\be  \label{[v]B0*W0=v^af}
G^0  =
\Big([v_1^0]_{\mathbb{B}_0}
\quad  \cdots \quad [v_r^0]_{\mathbb{B}_0} \Big)^{-T}
\Big([v_1^0]_{\mathbb{B}_1}
\quad  \cdots \quad [v_r^0]_{\mathbb{B}_1} \Big)^{T}.
\ee
If $ \big([v_1^0]_{\mathbb{B}_0}
\quad  \cdots \quad [v_r^0]_{\mathbb{B}_0} \big)$ is singular,
or nearly singular, we can apply small perturbations to $v_1^0,\ldots, v_r^0$.
Then after, find $\omega^0$ such that
$C+N(\omega^0) \approx G^0$, by solving it as a linear least squares problem.
Using such $\omega^0$ as a starting point,
we solve the polynomial system \reff{cmu:Mij:C+N*omg},
together with \reff{rand:aw=c} if $d>0$.

Suppose a parameter $\hat{\omega}$ satisfying \reff{cmu:Mij:C+N*omg}
is found as above. Let $\hat{G} :=G(\hat{\omega})$.
Apply the method in \S \ref{sbsc:polyf=0} to
get the complex zeros of $\varphi[\hat{G}]$.
If it has $r$ distinct zeros, say,
$\hat{v}_1,\ldots, \hat{v}_r$,
then there exist $\hat{\lmd}_1,\ldots, \hat{\lmd}_r$ satisfying
\be \label{leq:Vlmd=A}
\hat{\lmd}_1 (1,\hat{v}_1)^{\otimes m} +\cdots +
\hat{\lmd}_r (1,\hat{v}_r)^{\otimes m}  = \mc{F}.
\ee
For each $i=1,\ldots,r$, let
$\hat{u}_i = \sqrt[m]{\hat{\lmd}_i}(1, \hat{v}_i)$, then
\be \label{u:otm=F}
(\hat{u}_1)^{\otimes m} +\cdots + (\hat{u}_r)^{\otimes m}  = \mc{F}.
\ee
Because of round-off errors, the equation \reff{u:otm=F} may not be satisfied
very accurately. In such case, we can use $(\hat{u}_1,\ldots, \hat{u}_r)$
as a starting point and solve the nonlinear least squares problem \reff{nLS:A=u^m} again.
Usually, it can be solved very fast because
\reff{u:otm=F} is almost satisfied.

When $\varphi[\hat{G}]$ is defective (i.e., $\varphi[\hat{G}]$ has a repeated zero),
\reff{u:otm=F} cannot be guaranteed.
Suppose $\varphi[\hat{G}]$ has $r_0$ ($<r$) distinct zeros,
say, $\hat{v}_1,\ldots, \hat{v}_{r_0}$.
Choose general vectors $\hat{v}_{r_0+1},\ldots, \hat{v}_r$,
and then get $\hat{\lmd}_1,\ldots, \hat{\lmd}_r$
by solving \reff{leq:Vlmd=A} as a linear least square problem.
Formulate $\hat{u}_i$ same as above,
and then solve \reff{nLS:A=u^m}
by using $(\hat{u}_1,\ldots, \hat{u}_r)$ as a starting point.
Sometimes, this helps get a tensor decomposition.

Combining the above, we get the following algorithm.

\begin{alg} \label{alg:rc:dcmp}
A numerical method for symmetric tensor decompositions.  \noindent
\begin{itemize}

\item [{\bf Input:}] A tensor $\mc{F} \in \mt{S}^m(\cpx^{n+1})$,
an integer $r \geq \rank_S(\mc{F})$ (the default value of $r$
is given by \reff{AlxHrchFml}).

\item [{\bf Output:}] A tuple $(u_1,\ldots, u_r)$
satisfying \reff{s4:dcpA:u^m}.

%
%

\end{itemize}

\bit

\item [Step 0:] Parameterize $G$ as in \reff{fun:W(omg)}.
Let $d=r(n+1)-1-\dim \sig_r$.

\item [Step 1:] Solve \reff{nLS:A=u^m} with a random starting point.
Let $(u_1^0, \ldots, u_r^0)$ be a computed solution. If
$\mF = (u_1^0)^{\otimes m} + \cdots + (u_r^0)^{\otimes m}$,
output $(u_1^0, \ldots, u_r^0)$ and stop.

\item [Step 2:]
Let $v_i^0$ be such that $u_i^0 = \tau_i(1, v_i^0)$.
(If $(u_i^0)_0 = 0$, choose $v_i^0$ randomly.)
Compute $G^0$ as in \reff{[v]B0*W0=v^af}.
Find $\omega^0$ by solving $C+N(\omega) \approx G^0$
as a linear least squares problem.

\item [Step 3:] Solve \reff{cmu:Mij:C+N*omg}, together with
\reff{rand:aw=c} if $d>0$,
with $\omega^0$ a starting point,
to get a solution $\hat{\omega}$ by a numerical method.
Let $\hat{G} := G(\hat{\omega})$.

\item [Step 4:] Compute the complex zeros $\hat{v}_1,\ldots,\hat{v}_{r_0}$
of $\varphi[\hat{G}]$ by the method in \S\ref{sbsc:polyf=0}.
(If $r_0<r$, choose $\hat{v}_{r_0+1},\ldots, \hat{v}_{r}$ generically.)

\item [Step 5:] Determine $\hat{\lmd}_1,\ldots,\hat{\lmd}_r$ by solving \reff{leq:Vlmd=A}
as a linear least squares problem.

\item [Step 6:] Let $\hat{u}_i := (\hat{\lmd}_i)^{1/m}(1,\hat{v}_i)$
for $i=1,\ldots,r$. Solve \reff{nLS:A=u^m} with $(\hat{u}_1,\ldots, \hat{u}_r)$
as the starting point.
Output the solution as $(u_1,\ldots, u_r)$.

\eit

\end{alg}

In Step~0, if $G$ cannot be parameterized in the form \reff{fun:W(omg)}
(i.e., the set $\mathscr{G}(\mF)$ is empty),
then it is most likely that $\rank_S(\mF) > r$.
For such case, we need to increase the value of $r$.

In Algorithm~\ref{alg:rc:dcmp}, the major computation is in Step~3.
Most numerical methods cannot theoretically guarantee to find a solution
of \reff{cmu:Mij:C+N*omg}, together with \reff{rand:aw=c} if $d>0$.
However, in practice, we are often able to get one.

\begin{prop} \label{thm:STD:numalg}
Suppose a vector $\hat{\omega}$ satisfying \reff{cmu:Mij:C+N*omg},
together with \reff{rand:aw=c} if $d>0$, is found.
If the polynomial tuple $\varphi[\hat{G}]$ is nondefective,
then Algorithm~\ref{alg:rc:dcmp} produces a tensor decomposition for $\mF$.
\end{prop}
\begin{proof}
When $\varphi[\hat{G}]$ is nondefective, the generating polynomial tuple
$\varphi[\hat{G}]$ has $r$ distinct zeros.
The conclusion can be implied by Theorem~\ref{thm:dstcV=>lmd}.
\end{proof}

\begin{remark}  \label{rmk:cmu:okay}
In Algorithm~\ref{alg:rc:dcmp}, the nondefectiveness of
$\varphi[\hat{G}]$ cannot be always guaranteed.
However, the defectiveness does not happen very much.
As shown in \S\ref{sbsc:G:nodfc}, $\varphi[\hat{G}]$ is defective
only if $\hat{G}$ satisfies the additional equations
\be \label{disc:Mxi(W)=0}
\mbox{dis}\Big(M_{x_1}(\hat{G})\Big) = \cdots =
\mbox{dis}\Big(M_{x_n}(\hat{G})\Big) = 0.
\ee
So, $\varphi[\hat{G}]$ is usually nondefective.
This fact was observed in our numerical experiments.
To theoretically guarantee the nondefectiveness of
$\varphi[\hat{G}]$, we need to solve \reff{inV*Mxi*V=Di}.
Generally, \reff{inV*Mxi*V=Di} is much harder to solve,
because of the extra new matrix variables $V$ and $D_1, \ldots, D_n$.
\end{remark}

\subsection{Some comparisons}

Algorithms~\ref{alg:GR:algebraic} and \ref{alg:rc:dcmp}
have their own distinctive features.
Here, we make some comparisons between them.
\bnum

\item [1)] Algorithm~\ref{alg:GR:algebraic} uses algebraic methods to
solve the polynomial system \reff{cmu:Mij:C+N*omg},
together with \reff{rand:aw=c} if $d>0$.
Computing Gr\"{o}bner bases is usually required,
so Algorithm~\ref{alg:GR:algebraic} is often applied for small tensors.
On the other hand, Algorithm~\ref{alg:rc:dcmp}
uses numerical methods to solve the polynomial system.
So it can be applied for bigger tensors.
It often gets a solution,
but this cannot be mathematically guaranteed.

\item [2)] Algorithm~\ref{alg:GR:algebraic} assumes that
$\mF$ is a general point in $\sig_r$. Thus, its rank needs to be known
in advance. For general tensors in $\mc{S}^m(\cpx^{n+1})$,
their ranks are given by \reff{AlxHrchFml}.
On the other hand, Algorithm~\ref{alg:rc:dcmp}
only requires an upper bound $r \geq \rank_S(\mF)$.
A default value for $r$ can be generally chosen as in \reff{AlxHrchFml}.

\item [3)] For general $\mF \in \sig_r$, Algorithm~\ref{alg:GR:algebraic}
can find all rank decompositions if $d=0$,
and it can find a finite slice if $d>0$
(cf. Proposition~\ref{thm:symTD:prop}).
On the other hand, Algorithm~\ref{alg:rc:dcmp}
typically can only get one.
However, if Algorithm~\ref{alg:rc:dcmp} is repeatedly used
with random $a_i, b_i$ in \reff{rand:aw=c},
we may be able to get several different tensor decompositions.

%
%

\enum

As a summary, Algorithm~\ref{alg:GR:algebraic}
has more theoretical properties, but it usually works for small tensors;
Algorithm~\ref{alg:rc:dcmp} has less
theoretical properties, but it works more efficiently for bigger tensors.

\section{Computational experiments}
\label{sc:comp}
\setcounter{equation}{0}

This section reports numerical experiments for computing symmetric tensor decompositions.
The computation is implemented in 64-bit MATLAB R2012a,
on a Lenovo Laptop with Intel(R) Core(TM)i7-3520M CPU@2.90GHz and RAM 16.0G.
For Algorithm~\ref{alg:GR:algebraic}, the MATLAB
symbolic function {\tt solve} is used in Step~1.
For Algorithm~\ref{alg:rc:dcmp}, the MATLAB numerical
function {\tt lsqnonlin} is used to solve \reff{nLS:A=u^m} in Step~1 and Step~6,
and {\tt fsolve} is used in Step~3.

The output tuple $(u_1, \ldots, u_r)$
may not give an exact decomposition for $\mF$,
because of round-off errors.
We use the decomposition error, which is defined as
\[
\| (u_1)^{\otimes m} + \cdots + (u_r)^{\otimes m} - \mF \|,
\]
to measure the computational accuracy.
When the decomposition error is tiny (e.g., in the order of $10^{-10}$),
we cannot mathematically conclude that $\rank_S(\mF) \leq r$.
Indeed, we cannot even mathematically conclude that $\rank_{SB}(\mF) \leq r$.
Generally, we are not able to get exact decompositions. For such reasons,
all the claims about ranks of tensors in the examples
are modulo round-off errors. The presented tensors are all symmetric,
so their symmetric ranks are just called ranks, for convenience.

To present tensor decompositions, we display the vectors
$u_1, \ldots, u_r$ column by column, from left to right,
and from top to bottom (if one row block is not enough).
For neatness, we only display four decimal digits
for the real and imaginary parts.

Displaying all entries of a tensor usually
occupies a lot of space. To save space,
we display a tensor $\mF$ in three ways. The first one is to display
the vector $\mbox{uptri}(\mF)$ of upper triangular entries
in the lexicographical ordering, as in \reff{df:up:tri}.
The second way is to display $\mF$ as
the homogeneous polynomial $\mF(\tilde{x})$ as in \reff{df:F(tdx)}.
The third way is to give explicit formula for
$\mF_{i_1 \ldots i_m}$ if it exists.
The first way is convenient for small tensors,
the second one is good for tensors having a lot of zeros,
and the third one is good for tensors that are given by explicit formulae.

\subsection{Some technical tricks}

The following tricks are used for pre-processing of input tensors
and post-processing of output decompositions.

\smallskip
\noindent
{\bf Generic linear transformations} \,
The proposed methods assume the tensors are general.
In practice, it is hard to check such assumptions.
However, linear transformation is very useful for this purpose.
For a nonsingular matrix $Q \in \cpx^{(n+1)\times (n+1)}$,
define the linear transformation
$\mathscr{L}_Q$ from $\mt{S}^m(\cpx^{n+1})$ to itself, such that
\be \label{tsr:LnTrs:G}
\mathscr{L}_Q \Big( (u_1)^{\otimes m} + \cdots + (u_k)^{\otimes m} \Big) =
  (Q u_1)^{\otimes m} + \cdots + (Q u_k)^{\otimes m}.
\ee
Note that $\mathscr{L}_{Q^{-1}}(\mathscr{L}_Q (\mF) ) = \mF$ for all $\mF$.
If it is computed that
\[
\mathscr{L}_Q (\mF) =   (\tilde{u}_1)^{\otimes m}
+ \cdots + (\tilde{u}_r)^{\otimes m},
\]
then we can get a decomposition for $\mF$ as
\[
\mF = \mathscr{L}_{Q^{-1}}(\mathscr{L}_Q (\mF) )=
 (Q^{-1}\tilde{u}_1)^{\otimes m} + \cdots + (Q^{-1} \tilde{u}_r)^{\otimes m}.
\]
For a general $Q$, $\mathscr{L}_Q (\mF)$ is more likely
to be general than $\mF$ is.
In computation, we often choose $Q$ as a unitary matrix,
from the QR factorization of a randomly generated complex matrix.

\smallskip
\noindent
{\bf A length reduction process} \,
Algorithm~\ref{alg:rc:dcmp} only requires an upper bound $r$
for the rank. For general tensors in $\mc{S}^m(\cpx^{n+1})$,
their ranks are given by \reff{AlxHrchFml},
so Algorithm~\ref{alg:rc:dcmp} can produce rank decompositions.
For nongeneral tensors, however,
Algorithm~\ref{alg:rc:dcmp} may not produce rank decompositions.
Here, we propose a heuristic trick for reducing decomposition lengths.
Suppose we have computed that
\be \label{F=sum:ui:rc}
\mc{F} = (u_1^{old})^{\otimes m} + \cdots +  (u_r^{old})^{\otimes m}.
\ee
Then, $\rank_S(\mc{F}) \leq r$.
We can attempt to get a shorter decomposition than \reff{F=sum:ui:rc}.
Consider the optimization problem (with $\ell := r$)
\be \label{nls:A-sui:ell-1}
\min_{u_1, \ldots, u_{\ell-1} \in \cpx^{n+1} }
\quad  \| \mc{F} - \big((u_1)^{\otimes m} + \cdots
+ (u_{\ell-1})^{\otimes m} \big)  \|^2.
\ee
Reorder $u_1^{old},\ldots,u_r^{old}$ such that
$\|u_1^{old}\|_2 \geq \cdots \geq \|u_r^{old}\|_2$.
Using $(u_1^{old},\ldots, u_{\ell-1}^{old})$
as a starting point, we solve \reff{nls:A-sui:ell-1}
for an optimizer $(u_1^{new},\ldots, u_{\ell-1}^{new})$,
by using classical nonlinear optimization methods. If
\[
\mc{F} = (u_1^{new})^{\otimes m} + \cdots +  (u_{\ell-1}^{new})^{\otimes m},
\]
then we get a shorter decomposition and $\rank_S(\mc{F}) \leq \ell-1$.
Such attempting can be repeated,
until no further shorter decompositions can be found.
This process often produces rank decompositions.

\subsection{Examples for Algorithm~\ref{alg:GR:algebraic} }
\label{ssc:cmp:algbra}

We apply Algorithm~\ref{alg:GR:algebraic}
to randomly generated tensors. For cleanness of the presentation,
we round tensor entries to integers and display
their upper triangular entries.
In Step~1 of Algorithm~\ref{alg:GR:algebraic},
the MATLAB function {\tt solve} returns a set of symbolic objects
for the solutions to \reff{cmu:Mij:C+N*omg},
together with \reff{rand:aw=c} if $d>0$.
We first convert them to floating point numbers in double precision,
and then construct tensor decompositions numerically.

\begin{exm} \label{exmp:5.1}
Consider the tensor in $\mt{S}^3(\cpx^3)$
whose upper triangular entries are
\[
  -8 ,\quad    2  ,\quad    15   ,\quad   -7   ,\quad   17   ,\quad    7  ,\quad    17   ,\quad
     4   ,\quad    3    ,\quad  18.
\] \noindent The
generic rank is $4$. The dimension gap $d=2$.
Applying Algorithm~\ref{alg:GR:algebraic} with $r=4$,
we got $7$ decompositions, one of which is {\scriptsize
\begin{verbatim}
   1.8662 + 0.0000i   1.1593 + 2.0080i   0.7382 + 0.8006i   0.7382 - 0.8006i
   1.8396 + 0.0000i   0.7799 + 1.3508i  -1.6894 - 1.5455i  -1.6894 + 1.5455i
   2.3702 + 0.0000i  -0.7610 - 1.3180i  -0.7749 + 0.6705i  -0.7749 - 0.6705i
\end{verbatim} \noindent}It took a few seconds.
The decomposition error is around $10^{-13}$.
This tensor is randomly generated
(rounded to integers), so its rank is expected to be $4$.
For other tensors in $\mt{S}^3(\cpx^3)$ generated in the same way,
we also get $7$ rank decompositions.
\end{exm}

\begin{exm}  \label{exmp:5.2}
Consider the tensor in $\mt{S}^4(\cpx^3)$
whose upper triangular entries are
\[
  -7  ,\,  -2  ,\,    11   ,\,   18   ,\,   -7   ,\,   -1   ,\,    3   ,\,   -2   ,\,
   -15   ,\,   -9   ,\,  -13   ,\,   -14   ,\,  -11   ,\,  -13  ,\,    18 .
\] \noindent The generic rank is $6$.
The dimension gap $d=3$.
Algorithm~\ref{alg:GR:algebraic}, with $r=6$,
produced $8$ rank decompositions, one of which is{\tiny
\begin{verbatim}
   0.7514 - 0.1018i   0.7514 + 0.1018i   1.1580 - 1.1580i   1.1570 - 0.8269i   1.1570 + 0.8269i   1.5709 - 0.0000i
   1.0793 + 0.8609i   1.0793 - 0.8609i   0.0359 - 0.0359i   1.5058 + 0.8933i   1.5058 - 0.8933i  -1.3873 + 0.0000i
   0.7092 + 2.1400i   0.7092 - 2.1400i  -0.7264 + 0.7264i  -0.2329 - 0.1421i  -0.2329 + 0.1421i   1.4407 - 0.0000i
\end{verbatim} \noindent}It took a few seconds.
The decomposition error is around $10^{-14}$.
The catalecticant matrix has rank $6$,
so the tensor rank is $6$ by Lemma~\ref{lm:relF:ranks}.
This tensor is randomly generated (rounded to integers).
For other tensors in $\mt{S}^4(\cpx^3)$ generated in the same way,
we also get $8$ rank decompositions.
\end{exm}

\begin{exm}
Consider the tensor  in $\mt{S}^3(\cpx^4)$
whose upper triangular entries are {\small
\[
-20 ,\,  -17   ,\,   16   ,\,   10   ,\,   -4   ,\,   -8   ,\,    3   ,\,   -1  ,\,   -19   ,\,   -6   ,\,   -6
   ,\,   7   ,\,    9   ,\,  -13    ,\,   1   ,\,   17   ,\,   11  ,\,   -17   ,\,    7   ,\,    9 .
\] \noindent}The generic rank is $5$.
The dimension gap $d=0$.
Algorithm~\ref{alg:GR:algebraic} with $r=5$
produced the decomposition of length $5$: {\scriptsize
\begin{verbatim}
   1.4082 - 0.3142i   1.4082 + 0.3142i   1.2259 + 2.1233i   1.0401 + 1.4018i   1.0401 - 1.4018i
  -1.3913 + 0.6435i  -1.3913 - 0.6435i   0.4601 + 0.7968i   0.4605 + 1.1233i   0.4605 - 1.1233i
   1.6719 + 1.0000i   1.6719 - 1.0000i  -0.9877 - 1.7107i   1.3185 + 0.2770i   1.3185 - 0.2770i
  -0.7396 + 1.9762i  -0.7396 - 1.9762i  -0.6792 - 1.1763i  -1.7818 + 0.3468i  -1.7818 - 0.3468i
\end{verbatim}  \noindent}It took a few seconds.
The decomposition error is around $10^{-13}$.
This tensor is randomly generated (rounded to integers),
so its rank is expected to be $5$.
A general tensor in $\mt{S}^3(\cpx^4)$ has a unique rank decomposition,
by the Sylvester Pentahedral Theorem (cf.~\cite{OedOtt13}).
For other tensors in $\mt{S}^3(\cpx^4)$ generated in the same way,
Algorithm~\ref{alg:GR:algebraic} produces a unique decomposition.
\end{exm}

\begin{exm}
Consider the tensor $\mF \in \mt{S}^5(\cpx^3)$
whose upper triangular entries are {\small
\[
 13 ,\,  -15   ,\,   -2  ,\,   -18   ,\,    0   ,\,    6   ,\,   -4  ,\,   -19   ,\,   -1   ,\,   12   ,\,   13
  ,\,  -13   ,\,  -17   ,\,   -1   ,\,   16   ,\,  -11   ,\,   14   ,\,   -4   ,\,    11   ,\,   14  ,\,    19.
\] \noindent}The generic rank is $7$.
The dimension gap $d=0$.
Algorithm~\ref{alg:GR:algebraic} with $r=7$
produced the decomposition of length $7$ {\scriptsize
\begin{verbatim}
   0.6198 - 0.2015i   0.6198 + 0.2015i   1.6337 + 0.2854i   1.6337 - 0.2854i
   1.1238 + 1.0620i   1.1238 - 1.0620i   0.2365 + 1.5462i   0.2365 - 1.5462i
   0.2522 + 1.5162i   0.2522 - 1.5162i   1.1934 - 0.0441i   1.1934 + 0.0441i

   1.5652 + 0.5541i   1.5652 - 0.5541i   0.2086 - 0.0000i
  -0.6180 + 0.0916i  -0.6180 - 0.0916i  -1.5724 + 0.0000i
  -0.0397 + 1.1272i  -0.0397 - 1.1272i   1.1784 - 0.0000i
\end{verbatim} \noindent}It took a few seconds.
The decomposition error is around $10^{-12}$.
This tensor is randomly generated (rounded to integers),
so its rank is expected to be $7$.
A general tensor in $\mt{S}^5(\cpx^3)$
has a unique rank decomposition (cf.~\cite{OedOtt13}).
For other tensors in $\mt{S}^5(\cpx^3)$ generated in the same way,
Algorithm~\ref{alg:GR:algebraic} produces a unique decomposition.
\end{exm}

\subsection{Examples for Algorithm~\ref{alg:rc:dcmp} }
\label{ssc:comp:num}

For the tensors in this subsection, the decompositions
are computed by Algorithm~\ref{alg:rc:dcmp}.
We first present examples from the existing literature,
and then give examples in which
the tensor entries are given by explicit patterns
or randomly generated.

\begin{exm}  \label{exmp:land:5.5}
(\cite{Land12})
(i) Consider the tensor $\mF \in \mt{S}^4(\cpx^3)$ such that {\small
\[
\mF (x) = 81x_0^4 + 17x_1^4 + 626 x_2^4 -144 x_0x_1^2x_2 +216 x_0^3x_1
-108x_0^3x_2 + 216 x_0^2x_1^2 + 54 x_0^2x_2^2 +
\]
\[
96 x_0 x_1^3 -12 x_0x_2^3
-52 x_1^3x_2 + 174 x_1^2x_2^2 - 508 x_1x_2^3 + 72 x_0x_1 x_2^2 -216 x_0^2x_1x_2.
\] \noindent}Its rank is $2$.
When the default value $r=6$ is used, Algorithm~\ref{alg:rc:dcmp}
produced a decomposition of length $6$.
After the length reduction process, we got the rank decomposition
$\mF = (u_1)^{\otimes 4} + (u_2)^{\otimes 4}$ with
\[
u_1 = (0, 1, -5), \quad u_2 = (3,2,-1).
\]
(ii)
Consider the determinantal tensor $\mF \in \mt{S}^3( \cpx^{6} )$, i.e.,
\[
\mF(x) =
 - x_{5}\, {x_{1}}^2 + 2\, x_{1}\, x_{2}\, x_{4} - x_{3}\, {x_{2}}^2
 - x_{0}\, {x_{4}}^2 + x_{0}\, x_{3}\, x_{5}.
\]
It is known that $\rank_S(\mF) \geq 14$ (cf.~\cite[Theorem~9.3.2.1]{Land12}),
bigger than the generic rank $10$.
When $r=10$ is applied, Algorithm~\ref{alg:rc:dcmp}
produced a decomposition with error around $0.0058$.
(There are numerical troubles in solving the equations.)
When $r=11$ is applied, Algorithm~\ref{alg:rc:dcmp}
produced the decomposition of length $11$: {\tiny
\begin{verbatim}
   1.2359 + 0.5467i  -1.8366 + 0.6526i   0.0963 + 0.7613i   0.2952 + 1.3602i  -1.4792 + 0.4006i  -0.6387 + 1.6903i
   0.4523 + 0.9630i   0.4903 + 0.4622i  -1.4703 + 1.4378i  -2.4724 - 0.9869i  -1.2629 - 2.6269i   1.0544 - 2.2173i
  -1.0502 + 0.4300i   1.1265 - 1.3684i  -0.2843 - 1.4483i   0.4627 - 1.5513i   1.6149 - 0.1830i   0.1713 - 0.6364i
   0.5343 + 0.7474i   0.0074 - 0.7733i  -0.0470 + 1.3051i   0.5961 - 0.5373i   0.7353 + 0.7080i   1.0122 + 0.0752i
  -1.0428 - 0.8114i   0.5281 + 0.4770i  -0.0879 - 1.2599i   0.5852 + 0.2262i   0.1379 + 0.6921i  -0.7050 + 1.0983i
   0.3393 - 0.2303i  -0.3120 + 0.3630i   0.1766 + 0.3329i  -0.0228 + 0.4151i  -0.3364 + 0.1479i   0.0226 - 0.1473i

  -0.2274 + 1.6711i   2.2580 + 0.8493i  -1.5360 + 1.2919i   1.4677 - 0.4252i   0.5708 + 2.7533i
  -2.0336 - 0.4841i   0.1377 - 0.7207i   1.6532 + 1.0993i   0.1846 + 0.7815i   1.3681 + 0.0110i
   0.2045 - 0.8710i  -1.2248 - 0.0319i  -0.0292 - 1.2676i  -1.2821 - 0.2847i  -1.2838 - 1.5426i
   1.0508 - 0.3553i  -0.0450 + 0.9011i   0.6547 + 0.8055i   1.2157 + 0.5647i  -0.3345 - 0.1843i
   0.5834 + 0.7497i  -0.1419 + 0.3505i  -1.3341 - 0.1398i  -0.6288 + 0.0950i  -0.0307 + 0.4340i
  -0.1908 + 0.0042i   0.0936 - 0.1556i   0.1104 + 0.1668i   0.3394 + 0.0511i   0.2108 + 0.2644i
\end{verbatim} \noindent}It took a few seconds.
The decomposition error is around $10^{-13}$. \\
(iii) Consider the permanent tensor $\mF \in \mt{S}^3(\cpx^{6} )$, i.e.,
\[
\mF(x) =
x_{5}\, {x_{1}}^2 + 2\, x_{1}\, x_{2}\, x_{4} + x_{3}\, {x_{2}}^2
 + x_{0}\, {x_{4}}^2 + x_{0}\, x_{3}\, x_{5}.
\]
It is known that $\rank_S(\mF) \geq 12$ (cf.~\cite[Theorem~9.3.2.1]{Land12}),
bigger than the generic rank $10$.
Applying Algorithm~\ref{alg:rc:dcmp} with $r=10$,
we got the decomposition of length $10$: {\tiny
\begin{verbatim}
  -0.0404 + 0.1394i  -0.1895 - 0.1734i  -0.6525 - 0.2618i  -1.0300 + 0.3085i  -0.3315 + 0.9944i
  -0.2942 + 0.5660i   0.4021 + 0.4430i  -0.1272 + 0.4952i   0.0165 - 0.3695i  -0.4038 + 0.0154i
  -0.2450 + 0.7427i   0.2428 + 0.2709i   1.8617 + 0.4011i   2.7389 - 1.4007i   0.3827 - 2.9697i
  -0.2184 - 0.0602i  -0.1708 + 0.1502i   0.0923 + 0.0113i   0.0129 + 0.0630i  -0.0332 - 0.0065i
  -0.8474 + 1.4951i   0.8221 + 1.5331i   0.0568 - 0.6901i   0.2371 + 0.0380i   0.1770 - 0.4685i
  -0.3661 + 0.9878i   0.6529 + 0.6020i  -0.3873 + 0.9772i   0.1079 + 1.8126i   1.6032 + 0.9343i

  -0.2246 + 0.3292i   0.1790 + 0.3603i  -0.1363 - 0.7170i   0.7674 - 0.4874i   0.0714 - 0.7727i
   0.2580 + 0.0409i   0.2429 - 0.1631i  -0.5041 + 0.1659i   0.0787 + 0.2020i   0.4545 - 0.2890i
  -1.2741 + 1.7287i   0.8934 + 1.9753i  -0.1586 + 1.3532i  -1.5725 + 0.3272i  -0.6530 + 1.2214i
  -0.0201 + 0.1446i  -0.0351 - 0.0399i   0.0418 + 0.0870i   0.0124 - 0.0505i   0.1085 + 0.0491i
   0.6907 + 0.1353i   0.6700 - 0.4153i   0.4954 - 0.0066i   0.7072 + 0.3046i   0.2027 - 0.5534i
  -1.7998 + 2.2648i   1.1061 + 2.7028i   1.1682 + 1.6344i  -1.5470 + 2.3971i   0.9532 + 2.3314i
\end{verbatim} \noindent}It took a few seconds.
The decomposition error is around $10^{-11}$.
\end{exm}

For the determinantal tensor in Example~\ref{exmp:land:5.5}(ii),
its true rank is at least $14$,
bigger than $10$ (the generic one of the space $\mt{S}^3(\cpx^6)$).
Algorithm~\ref{alg:rc:dcmp} did not get an accurate numerical
decomposition of length $10$.
The equation \reff{cmu:Mij:C+N*omg}, for $r=10$,
is ill-conditioned, and it cannot be solved accurately.
When $r$ is increased to $11$, the equation \reff{cmu:Mij:C+N*omg}
is better-conditioned, and can be solved easier.
The mathematical reasons for such numerical issues
are not known to the author.

\begin{exm} (\cite{BCMT10})
(i) Consider the tensor $\mF \in \mt{S}^5 ( \cpx^3 )$ such that {\smaller \smaller
\[
\mF(x) = (x_0+2x_1+3x_2)^5 + (x_0-2x_1+3x_2)^5 + \frac{1}{3} (x_0-12x_1-3x_2)^5
+\frac{1}{5} (x_0 + 12 x_1 - 13x_2)^5.
\] \noindent}Its rank is $4$.
When the default value $r=7$ is used, Algorithm~\ref{alg:rc:dcmp}
produced a decomposition of length $7$.
After the length reduction process, we got a rank decomposition, which is {\tiny
\begin{verbatim}
   0.2240 + 0.6893i  -0.6494 + 0.4718i  -0.8090 - 0.5878i  -0.8090 - 0.5878i
   2.6876 + 8.2717i   7.7932 - 5.6621i  -1.6180 - 1.1756i   1.6180 + 1.1756i
  -2.9116 - 8.9610i   1.9483 - 1.4155i  -2.4271 - 1.7634i  -2.4271 - 1.7634i
\end{verbatim} \noindent}It took a few seconds.
The decomposition error is around $10^{-10}$.
\\
(ii) Consider the tensor $\mF \in \mt{S}^4 ( \cpx^3 )$ such that {\smaller
\[
\mF(x) =79x_0x_1^3 + 56x_0^2x_2^2 + 49 x_1^2x_2^2+4x_0x_1x_2^2 + 57x_0^3x_1
\] \noindent}Its rank is $6$.
When the default value $r=6$ is applied, Algorithm~\ref{alg:rc:dcmp}
produced the rank decomposition: {\tiny
\begin{verbatim}
   0.8027 - 0.8235i   0.7864 + 0.1426i  -1.0389 + 1.0040i  -0.0325 + 1.4050i  -0.4618 + 0.5687i   0.0741 + 1.2199i
  -1.2113 + 1.2830i  -1.6084 - 0.0089i   1.2660 - 1.1355i   0.0542 + 1.5700i   0.3944 + 0.4603i  -0.1015 + 1.5205i
  -1.0647 - 1.7566i   0.4808 - 2.0482i  -0.6832 - 0.3842i  -0.0450 + 0.9696i  -0.9270 - 1.1481i   0.0658 - 1.1902i
\end{verbatim} \noindent}It took a few seconds.
The decomposition error is around $10^{-12}$.
\end{exm}

\begin{exm}
Consider the tensor $\mF \in \mt{S}^3( \cpx^4 )$ given as
\[
\mF_{i_1 i_2 i_3} =  (1+i_1i_2i_3)^{-1}
 \quad (0 \leq i_1, i_2, i_3 \leq 3).
\]
Applying Algorithm~\ref{alg:rc:dcmp} with the default value $r=5$,
we got the decomposition: {\tiny
\begin{verbatim}
  -0.0000 - 0.0000i  -0.5000 - 0.8660i  -0.0000 - 0.0000i  -0.0000 - 0.0000i   0.0000 - 0.0000i
  -0.8792 + 0.0561i  -0.5000 - 0.8660i   0.0518 - 0.1252i  -0.3898 - 0.3393i  -0.2517 - 0.3888i
  -0.9609 + 0.0129i  -0.5000 - 0.8660i  -0.1473 + 0.1377i  -0.2421 - 0.0960i  -0.2566 - 0.1486i
  -0.9894 - 0.0118i  -0.5000 - 0.8660i  -0.2603 + 0.2142i  -0.2163 + 0.0472i  -0.1944 - 0.0513i
\end{verbatim} \noindent}It took a few seconds.
The decomposition error is around $10^{-14}$.
The catalecticant matrix has rank $4$.
We are not sure whether the above gives a rank decomposition or not.
\end{exm}

\begin{exm}
Consider the tensor $\mF \in \mt{S}^3( \cpx^5 )$ given as
\[
\mF_{i_1 i_2 i_3} = i_1i_2i_3-i_1-i_2 -i_3
 \quad (0 \leq i_1, i_2, i_3 \leq 4).
\]
The catalecticant matrix has rank $2$.
Applying Algorithm~\ref{alg:rc:dcmp} with $r=2$,
we got the decomposition: {\tiny
\begin{verbatim}
  -0.6874 + 0.3969i  -0.6874 - 0.3969i
  -1.0842 - 0.2905i  -1.0842 + 0.2905i
  -1.4811 - 0.9779i  -1.4811 + 0.9779i
  -1.8779 - 1.6652i  -1.8779 + 1.6652i
  -2.2748 - 2.3526i  -2.2748 + 2.3526i
\end{verbatim} \noindent}It took a few seconds.
The decomposition error is around $10^{-14}$.
The above gives a rank decomposition by Lemma~\ref{lm:relF:ranks}.
\end{exm}

\begin{exm}
Consider the tensor $\mF \in \mt{S}^4( \cpx^4 )$ given as
\[
\mF_{i_1 i_2 i_3 i_4} =  \tan( i_1 i_2 i_3 i_4  )
 \quad (0 \leq i_1, i_2, i_3, i_4 \leq 4).
\]
With the default value $r=10$, Algorithm~\ref{alg:rc:dcmp}
produced a decomposition of length $10$. After the length reduction,
we got the decomposition of length $6$: {\tiny
\begin{verbatim}
   0.0000 - 0.0000i   0.0000 - 0.0000i   0.0000 - 0.0000i  -0.0000 - 0.0000i  -0.0000 - 0.0000i  -0.0000 - 0.0000i
   1.4054 + 2.4745i   2.1010 + 1.3045i  -0.0297 + 2.3633i   0.5016 - 2.3944i  -0.4230 + 0.3744i   0.5743 + 0.1277i
   0.2365 - 2.9247i  -1.4229 - 2.0011i   1.6962 - 1.5666i  -1.1281 + 1.4019i  -1.2226 + 1.2100i  -1.4871 + 0.0508i
  -0.3392 + 5.3488i   2.4758 + 3.6000i  -3.0429 + 3.3063i   2.5311 - 2.2554i   0.0084 + 0.3644i  -1.2250 + 0.0647i
\end{verbatim} \noindent}It took a few seconds.
The decomposition error is around $10^{-12}$.
The catalecticant matrix has rank $6$,
so the above gives a rank decomposition by Lemma~\ref{lm:relF:ranks}.
\end{exm}

\begin{exm}
Consider the tensor $\mF \in \mt{S}^4( \cpx^5 )$ given as
\[
\mF_{i_1 i_2 i_3 i_4} = \sin (i_1+i_2+i_3+i_4) + \cos (i_1 i_2 i_3 i_4)
 \quad (0 \leq i_1, i_2, i_3, i_4 \leq 4).
\]
Its catalecticant matrix has rank $12$.
Applying Algorithm~\ref{alg:rc:dcmp} with $r=12$, we got
the decomposition of length $12$: {\tiny
\begin{verbatim}
   0.0000 - 0.0000i  -0.0000 - 0.0000i   0.0000 + 0.0000i   0.7769 - 0.3218i   0.0000 - 0.0000i  -0.0000 - 0.0000i
  -0.2682 - 0.9708i  -0.1367 - 0.0680i  -0.0096 - 0.0050i   0.6905 + 0.4799i   0.9205 + 0.4450i   0.5288 - 0.3722i
  -0.7311 - 0.4860i  -0.2521 + 0.3927i   0.1162 + 0.0270i  -0.0307 + 0.8403i  -0.3466 + 0.3207i  -0.7863 + 0.6576i
   0.1919 - 0.0396i  -0.1093 + 0.5180i  -1.2379 + 0.0113i  -0.7237 + 0.4282i   0.0571 + 0.8127i   0.1489 - 0.1645i
  -0.7590 + 0.0976i   1.1092 + 0.8967i  -1.3191 + 0.0290i  -0.7513 - 0.3776i  -0.0794 - 0.2031i   1.1553 - 1.1422i

  -0.0000 - 0.0000i  -1.0000 + 0.0000i   0.0000 - 0.0000i   0.0000 - 0.0000i   0.3218 - 0.7769i   0.0000 - 0.0000i
   0.2109 - 0.4326i  -1.0000 + 0.0000i  -0.4919 + 0.3044i  -0.5844 - 0.9348i  -0.4799 - 0.6905i  -1.0052 - 0.3150i
  -0.1698 + 0.8876i  -1.0000 + 0.0000i  -0.6169 + 0.5712i  -0.4604 + 0.1266i  -0.8403 + 0.0307i  -0.4973 - 0.7259i
  -0.0010 + 0.0363i  -1.0000 - 0.0000i  -1.0080 + 0.9616i  -0.9528 - 0.1489i  -0.4282 + 0.7237i  -0.0600 + 0.1927i
  -0.1281 - 1.6921i  -1.0000 - 0.0000i  -0.7208 + 0.9866i  -0.0045 + 0.2068i   0.3776 + 0.7513i   0.0484 - 0.7334i
\end{verbatim} \noindent}It took a couple of seconds.
The decomposition error is around $10^{-14}$.
The above gives a rank decomposition by Lemma~\ref{lm:relF:ranks}.
\end{exm}

\begin{exm}
Consider the tensor in $\mt{S}^3( \cpx^7 )$ whose
upper triangular entries are {\scriptsize
\[
49  ,\,  52   ,\,   37   ,\,  -37  ,\,   -50  ,\,   -93  ,\,    89  ,\,  -147  ,\,    53  ,\,
 -20   ,\,   -1   ,\,   26  ,\,   -16   ,\,  -29   ,\,  9  ,\,   -17   ,\,   75   ,\,  -38 ,
 -17  ,\,   -77   ,\,   71 ,\,
\]
\[
 99   ,\,  4   ,\,   61   ,\,  -23  ,\,   -42  ,\,   -17   ,\,   11  ,\,
  -87  ,\,   -45  ,\,    -8   ,\,  116   ,\,  -46   ,\,  -38  ,\,   -48  ,\,  -117 , \,
-40  ,\,    65  ,\,    77  ,\,  -3   ,\,    96 ,\,
\]
\[
80   ,\,  -93   ,\,   12   ,\,    5  ,\,   -31 ,
 -44  ,\,   -39   ,\,   10  ,\,   -79   ,\,   64   ,\,   -5  ,\,   137  ,\,    -7
 -40  ,\,    36  ,\,   -30   ,\,    9  ,\,  59   ,\,   41 ,\, 44  ,\, 37 ,\,
\]
\[
-17   ,\,   -8 ,
 -95   ,\,    5  ,\,    31   ,\,   12  ,\,    66  ,\,    21   ,\,  106   ,\,  -12, \,
 40   ,\,    6   ,\,   78   ,\,  -16   ,\,   96   ,\,   52   ,\,   15   ,\,   -4  ,\,
 -53   ,\,   10   ,\,    3   ,\,   -8.
\] \noindent}For
the default value $r=12$, Algorithm~\ref{alg:rc:dcmp}
produced the decomposition:{\tiny
\begin{verbatim}
   3.3887 - 0.9935i   0.1743 - 1.9932i  -1.3754 + 0.1574i   1.7164 + 1.2273i  -0.2199 - 0.6368i   0.1687 + 0.2668i
  -0.6328 + 5.3674i   0.3875 - 2.6898i   0.4848 - 1.5815i   3.1488 - 1.4028i   1.2996 - 1.6498i   2.4370 - 1.1857i
  -0.1043 + 0.0487i  -1.8504 + 1.8720i   1.5402 - 4.8430i  -0.8872 + 1.6694i  -4.5402 - 1.3199i  -2.7707 + 2.5037i
   1.0096 + 3.5096i  -2.6621 + 1.4238i   0.1103 - 0.4865i  -2.2066 + 1.2238i   0.2269 + 0.5137i   2.7036 + 4.1404i
  -2.3543 + 0.8867i  -1.0351 - 1.0243i   1.2195 - 1.8089i   0.3480 - 1.2185i  -1.8558 - 0.0594i  -0.4570 + 3.0827i
  -1.0119 - 0.9969i   2.2954 - 0.7927i  -3.7602 + 0.2418i  -2.2716 - 1.8641i   1.3051 - 3.5072i   0.5335 + 0.3160i
   1.5778 - 0.7354i   0.0601 - 2.2870i  -0.0562 + 1.7575i   0.6315 - 1.1591i   0.0592 + 0.5537i  -2.3464 - 0.3364i

   2.5666 + 2.5520i  -0.1979 + 0.9444i  -2.7901 - 1.2362i  -2.9522 + 3.1287i  -0.2476 + 3.1839i  -1.8626 - 1.1029i
  -2.9476 - 1.3906i   1.2410 + 0.8104i   2.0747 + 1.0738i   3.5610 + 2.1201i   3.5116 + 1.2409i   3.6480 + 4.3116i
  -0.5535 - 1.6854i  -1.5364 + 0.0220i   0.5836 + 0.9219i   1.1947 + 0.0594i  -2.8377 - 2.7227i  -0.2092 + 2.3056i
  -0.9648 + 4.2704i  -4.2003 + 2.4566i   0.0692 + 4.5040i   2.5234 + 0.5384i  -0.1643 - 1.0481i   0.6575 + 0.0574i
   0.9150 + 1.3780i  -2.3690 + 2.7505i  -0.8278 - 0.8274i   3.3398 - 1.2667i  -0.8625 - 3.4098i  -0.4542 + 0.7497i
  -0.4286 + 0.8020i  -0.7728 + 1.9996i   0.3706 + 0.5754i  -1.1274 - 2.1393i   1.0650 + 0.8385i  -0.8726 + 0.9087i
  -1.4644 - 2.0376i  -0.2382 + 2.4795i   1.8640 + 1.7468i  -0.4149 + 0.9290i  -1.9110 + 0.1033i  -0.1823 - 2.1603i
\end{verbatim} \noindent}It took a few minutes.
The decomposition error is around $10^{-12}$.
The tensor was generated randomly (rounded to integers),
so its rank is expected to be $12$.
For other randomly generated tensors in $\mt{S}^3(\cpx^7)$,
we have similar computational results.
\end{exm}

\begin{exm}
Consider the tensor in $\mt{S}^4( \cpx^5 )$
whose upper triangular entries are:{\scriptsize
\[
-75 ,\,  -12  ,\,   -9  ,\,  -60   ,\,  43  ,\,   30   ,\,  95   ,\, -31   ,\, -12   ,\,  23   ,\,
30  ,\,  -17  ,\,  -29  ,\,   65  ,\,   54  ,\,  -17 ,\,   18    ,\,  3  ,\,
\]
\[
-102 , \, -16  ,\,  -10   ,\,  48  ,\,  -14   ,\,  35   ,\,   6  ,\,   44  ,\,  -67   ,\,   5   ,\,  -3  ,\,
10  ,\,   13  ,\,   19  ,\,  9   ,\, -15  ,\, -126   ,\,  54  ,\,
\]
\[
-38  ,\,   -1  ,\,
-22  ,\,  -40   ,\, -25   ,\,  17  ,\,  -27  ,\,  -35  ,\,  -75   ,\,  31   ,\,  74   ,\, -43 ,\,
    38   ,\,  46  ,\,    6   ,\,  38   ,\,  30  ,\,  -64   ,\,
\]
\[
18   ,\,  24   ,\,  14  ,\,
-20   ,\, -17   ,\, -92  ,\,  -50   ,\, -31   ,\, -46  ,\,   67 ,\,
    36  ,\,  -26  ,\,   17  ,\,  -56    ,\, -7  ,\,  -18.
\] \noindent}The
tensor was randomly generated (rounded to integers). For
the default value $r=15$, Algorithm~\ref{alg:rc:dcmp}
produced the decomposition: {\tiny
\begin{verbatim}
   1.7766 - 0.7962i  -0.0989 + 1.5672i   0.4586 + 1.2843i  -1.0824 - 2.2054i  -0.9173 + 0.0833i
   0.5902 - 1.7985i   0.7956 - 0.5731i   2.1843 - 0.0837i  -0.2960 + 1.0261i   0.9095 + 1.2630i
  -0.3040 + 0.4909i   3.5389 - 1.4917i  -1.5819 - 2.1610i  -1.2090 + 3.3445i   0.8739 - 3.5943i
   2.1547 + 0.3346i   0.3134 - 0.1395i   0.1921 - 0.5282i   0.4854 + 1.1021i  -0.5342 - 0.1079i
  -0.6948 + 0.3296i  -2.0563 + 1.5891i  -1.4849 + 3.6039i  -0.5785 - 1.9694i  -2.6814 + 2.0639i

  -0.9571 - 0.8367i  -0.9046 - 2.7956i  -1.1199 + 0.5681i   1.6946 - 2.8898i   1.2026 - 1.4611i
   2.5902 - 0.2139i  -0.8625 - 1.1986i  -1.4064 - 2.0862i  -0.0069 - 1.6966i  -2.0682 - 0.9022i
  -0.8195 - 1.9443i   0.5415 + 0.5792i   0.4608 + 1.0577i  -0.5489 + 0.9552i   0.3738 - 1.7060i
  -0.7524 + 0.4986i  -0.8816 - 0.7419i  -0.9164 + 0.3831i   1.1807 - 1.2444i   1.0867 - 0.0443i
   1.6401 + 1.4332i   0.4170 + 4.3105i  -1.5277 - 2.8687i  -1.9673 + 3.0638i  -1.9649 - 0.8143i

  -1.0529 - 1.3452i  -0.8754 - 2.4256i  -0.0504 - 0.2344i  -2.4363 - 1.0828i   1.0572 - 1.5768i
   0.1519 - 2.2219i   0.8495 - 1.6074i   1.5621 - 1.8827i   0.2592 - 0.3617i   2.3835 - 0.2981i
   0.3655 - 0.7656i   0.9260 - 0.6683i  -1.8452 + 2.0499i   1.8011 - 0.9841i  -2.3706 + 1.0560i
   0.4716 + 1.4494i  -1.0883 + 0.9951i  -1.6793 + 1.5497i   1.0282 + 1.1317i   0.8274 + 0.2019i
  -1.4550 + 0.6641i   0.8709 - 1.4441i  -1.1551 + 0.9121i  -0.8861 - 2.1757i  -1.0959 - 1.8642i
\end{verbatim} \noindent}It took a couple of seconds.
The decomposition error is around $10^{-11}$.
The catalecticant matrix has rank $15$,
so the above gives a rank decomposition
by Lemma~\ref{lm:relF:ranks}.
For other randomly generated tensors in $\mt{S}^4(\cpx^5)$,
we have similar computational results.
\end{exm}

\begin{exm}
Consider the tensor in $\mt{S}^5 ( \cpx^4 )$ whose upper triangular entries are {\scriptsize
\[
17 ,\,  -13   ,\,   29  ,\,    20  ,\,   -30   ,\,   -5   ,\,  -83   ,\,  -10   ,\,
93  ,\,   -69   ,\,   52   ,\,  128   ,\,   67   ,\,   40   ,\,
 -6  ,\,    66  ,\,  46   ,\,   -8   ,\,  -10  ,\,
\]
\[
 -108  ,\, -33    ,\,  80  ,\,   -16   ,\,   47  ,\,   -19  ,\,   -23   ,\,   38  ,\,    66   ,\,
 54  ,\,   -14   ,\,    0   ,\,   22   ,\,  39   ,\,   77   ,\,   28   ,\,   55  ,\,
 -59   ,\,   73  ,\,
\]
\[
81  ,\,   67  ,\,    20   ,\,    9    ,\,
 -33   ,\,  -33  ,\,   -32   ,\,   10  ,\,   -31   ,\,   -1   ,\,   44   ,\,   31   ,\,
 -67  ,\,   -37  ,\,    34   ,\,    6   ,\,   81   ,\,  -55.
\] \noindent}For
the default value $r=14$, Algorithm~\ref{alg:rc:dcmp}
produced the decomposition: {\tiny
\begin{verbatim}
  -0.4859 - 0.5911i  -0.6703 + 1.0296i  -1.7767 + 0.8556i  -1.4424 + 1.3130i   2.0146 - 0.8757i
   1.8808 + 0.0976i   1.1844 - 0.0241i   0.6504 + 1.6116i   2.0375 + 0.8269i   2.2090 - 0.5356i
  -1.3334 - 0.1266i  -1.2351 - 0.2072i  -0.1653 - 1.4772i   0.1740 + 1.6152i   1.2336 + 0.6581i
  -2.6869 - 0.3499i  -2.3412 + 1.8382i   1.4152 - 0.6755i  -0.6329 - 0.8693i   0.9572 - 0.0797i

  -1.6874 - 0.6179i  -1.4593 - 0.4921i  -1.7642 - 0.7340i  -1.7982 - 0.2974i  -1.5568 + 0.7872i
   0.7319 - 1.6435i   1.2276 + 0.4577i   1.5966 + 0.2339i   1.7420 + 0.8219i   1.5418 - 0.0386i
  -0.0404 + 1.5668i   1.6082 - 1.1321i   0.1856 - 0.8780i  -0.6215 + 1.1713i   1.5319 + 1.1126i
   1.6182 + 1.0162i  -0.3730 - 1.7918i   1.1269 - 1.4395i   0.9538 - 0.9343i   0.0213 + 1.9383i

  -1.6628 + 0.5729i  -1.3538 - 1.3348i  -2.0749 + 0.0471i   2.1352 + 0.6442i
   2.1360 - 0.6526i   1.9225 - 0.8532i   0.5884 - 0.1493i   2.1692 + 0.6314i
  -0.2769 - 0.5100i   0.1933 - 1.2516i   0.8709 + 0.3253i   1.3007 - 0.6285i
   1.1450 + 1.3432i  -0.8092 + 1.0486i  -2.7097 + 0.4264i   0.7922 + 0.2493i
\end{verbatim} \noindent}It took a couple of seconds.
The decomposition error is around $10^{-12}$.
The tensor was randomly generated (rounded to integers),
so the above is expected to be a rank decomposition.
For other randomly generated tensors in $\mt{S}^5(\cpx^4)$,
we have similar computational results.
\end{exm}

\begin{exm}
Consider the tensor in $\mt{S}^6 ( \cpx^4 )$
whose upper triangular entries are{\scriptsize
\[
    1 ,\,  -35  ,\,   -15   ,\,  -75  ,\,    98  ,\,   -52   ,\,    3  ,\,   -15   ,\,  -92   ,\,  -51   ,\,
    28  ,\,     9   ,\,   25   ,\,  -12  ,\,   -12   ,\,   -6   ,\,   54  ,\,   -23  ,\,   -61  ,\,   -13  ,\,
    15  ,\,
\]
\[
 -4  ,\,   -32  ,\,   -10  ,\,   -47  ,\,   -59   ,\,    9   ,\,  -12  ,\,    94   ,\,  -87  ,\,
    -34  ,\,    -6   ,\,   26  ,\,    72   ,\,  -34  ,\,  -113  ,\,   122   ,\,   68  ,\,   -19   ,\,  -13  ,\,
    77  ,\,
\]
\[
 68  ,\,    32  ,\,    20  ,\,   -46  ,\,    61  ,\,   -14  ,\,   -60  ,\,    58  ,\,   -14   ,\,
 57   ,\,   -4  ,\,   -55   ,\,  -18  ,\,   -15   ,\,   13   ,\,  -36   ,\,    3   ,\,  -16  ,\,   -26  ,\,
 77  ,\,  -82  ,\,  -17  ,\,
\]
\[  -19   ,\,   81  ,\,   -32   ,\,   45  ,\,    68  ,\,   -91   ,\,    3   ,\,
    -67  ,\,   -41   ,\,   32   ,\,   49  ,\,   -16   ,\,  -26   ,\,   47  ,\,   -20  ,\,    -18  ,\,    70  ,\,
   -28   ,\,   22   ,\,   25  ,\,   -42.
\] \noindent}For
the default value $r=21$, Algorithm~\ref{alg:rc:dcmp}
produced the decomposition: {\tiny
\begin{verbatim}
  -1.4596 + 0.5136i   2.0603 + 0.5095i   1.1222 + 0.0638i   0.8154 - 1.7817i  -0.2524 + 0.7907i   1.0125 - 0.5051i
   0.3135 + 0.7200i  -0.6900 + 1.1717i   0.2507 + 1.6619i   0.1791 + 0.0394i  -1.2794 + 0.9613i  -2.0645 - 0.1349i
  -1.4133 + 0.1344i   0.6915 - 0.2367i   0.5306 + 1.7345i  -0.0531 + 1.6124i   1.7085 + 0.4062i   1.3432 + 0.6793i
   0.7936 + 1.8248i  -0.5297 + 1.2167i   0.6811 - 0.4799i   1.9304 - 0.4848i  -1.2445 + 1.1681i   0.0755 - 1.6539i

  -0.1305 - 1.3478i  -1.3394 + 0.5479i  -0.3195 + 1.6664i   1.4475 + 0.9899i  -1.3831 - 1.0481i  -1.2866 + 1.3324i
  -1.1321 - 0.8737i  -0.7956 - 0.8972i  -0.5081 + 0.9225i  -1.0656 + 0.3921i   0.3516 - 1.4076i   1.1224 - 0.1997i
   1.7508 - 0.3281i   0.9495 - 0.4047i   0.5261 + 0.2362i   0.8792 + 0.3875i   0.5770 - 0.6563i   1.5912 - 0.8020i
  -1.4245 - 1.0803i  -0.4005 + 2.1648i   0.5388 + 2.1204i   0.9179 + 1.1780i   1.6059 + 0.1686i   0.9348 + 1.1892i

  -0.4193 - 0.7312i   1.2306 - 1.1451i   0.7843 + 0.3831i   0.7619 - 1.0376i  -1.2090 + 0.1338i  -1.0915 - 1.0439i
  -0.1698 + 2.1097i  -0.0778 - 1.3340i  -2.0102 + 0.3982i  -0.0205 - 1.9106i   0.9346 + 1.8949i   0.7062 + 0.0262i
   1.2759 + 1.0468i  -1.0679 + 0.8113i   1.3695 - 0.5441i   0.8400 + 0.8760i   1.5698 + 0.2902i  -0.0079 + 1.3733i
   0.8825 + 1.6556i   2.1060 + 0.4846i  -0.2209 + 1.5111i  -0.7921 + 1.1630i   1.6861 + 1.0216i   1.5312 + 1.3715i


  -0.5390 + 1.8534i   0.9852 + 0.7079i  -1.5235 - 0.2267i
   1.4676 - 1.1163i   1.4815 - 0.6813i  -1.5642 - 0.6497i
  -0.3530 - 0.6155i   1.7008 - 0.1266i   0.6668 + 1.1214i
   0.5705 + 0.6937i  -0.3041 + 0.8836i   1.3072 + 0.4241i
\end{verbatim} \noindent}It took a few minutes.
The decomposition error is around $10^{-12}$.
The tensor was generated randomly (rounded to integers),
so its rank is expected to be $21$. The catalecticant matrix has rank $20$.
The above is likely a rank decomposition.
For other randomly generated tensors in $\mt{S}^6(\cpx^4)$,
we have similar computational results.
\end{exm}

\begin{exm}
(random tensors)
This example explores the performance of Algorithm~\ref{alg:rc:dcmp}
for random symmetric tensors. We generate $\mF$
such that each entry $\mF_{i_1 \ldots i_m}$
is a random complex number (up to symmetry), whose real and imaginary parts
obey Gaussian distributions. In MATLAB, this can be done by the command
{\tt randn + sqrt(-1)*randn}. Such tensors
are expected to have generic ranks given by \reff{AlxHrchFml}.
We apply Algorithm~\ref{alg:rc:dcmp} with the default values for $r$.
For the pairs of $(n+1,m)$ in Table~\ref{time:rand:tsr:grank},
we generate $50$ instances of random tensors in $\mt{S}^m(\cpx^{n+1})$.
(If it appears, the symbol $\star$ means that only $20$ instances were generated,
because the computation takes longer time.)
For all the instances, Algorithm~\ref{alg:rc:dcmp} successfully
computed decompositions of desired lengths.
The decomposition errors are in the order of $10^{-10}$. For each pair $(n+1,m)$,
we report the average time (in seconds) consumed by the computation.
The computational results are summarized in
Table~\ref{time:rand:tsr:grank}.
For such tensors, we can get their rank decompositions efficiently.
\begin{table}[htb] \small
\caption{Computational results for rank decompositions
of random tensors in $\mt{S}^m(\cpx^{n+1})$. The $r$ is the generic rank,
and {\tt time} (in seconds) is the average of the consumed time.}
\label{time:rand:tsr:grank}
\begin{tabular}{||l|c|r||l|c|r||l|c|r||} \hline
(n+1,m) & $r$ & time & (n+1,m) & $r$ & time & (n+1,m) & $r$ & time  \\ \hline
(3, 3) & 4  &   0.1  &  (8, 3)$\star$ &  15 &   1799.6   & (3, 5) &  7  &  0.7 \\  \hline
(4, 3) & 5  &    0.3    &  (3, 4) & 6  &  0.1   & (4, 5) &  14  &  40.8  \\  \hline
(5, 3) & 8  &  0.7 &  (4, 4) & 10  &   1.1   & (5, 5)$\star$  & 26  & 2168.3  \\  \hline
(6, 3) & 10  &  3.5   &  (5, 4) & 15 &  9.9   & (3, 6) & 10  &  1.1   \\  \hline
(7, 3) & 12  & 361.1  &  (6, 4) & 21  &  575.5   & (4, 6) & 21  &  175.3  \\  \hline
\end{tabular}
\end{table}
\end{exm}

\noindent
{\bf Acknowledgement}
The author would like to thank Lek-Heng Lim, Luke Oeding
and two anonymous referees for the useful comments.
The research was partially supported by the NSF grants
DMS-0844775 and DMS-1417985.


\begin{thebibliography}{100}


\bibitem{AlxHirs95}
J.~Alexander and A.~Hirschowitz.
Polynomial interpolation in several variables.
{\em J. Algebraic Geom.}\, 4(1995), pp. 201-22.



\bibitem{BalBer12}
E.~Balllico and A.~Bernardi.
Decomposition of homogeneous polynomials with low rank.
{\it Math. Z.} \, 271, \, 1141-1149, \, 2012.



\bibitem{BerGimIda11}
A.~Bernardi, A.~Gimigliano and M.~Id\`{a}.
Computing symmetric rank for symmetric tensors.
{\it Journal of Symbolic Computation} \, 46, (2011), 34-53.



\bibitem{BCMT10}
J.~Brachat, P.~Comon, B.~Mourrain and E.~Tsigaridas. Symmetric
tensor decomposition. {\em Linear Algebra Appl.} \, 433, no. 11-12,
1851-1872, 2010.

%
%

\bibitem{BucBuc10}
W.~Buczy\'{n}ska and J.~Buczy\'{n}ski.
Secant varieties to high degree Veronese reembeddings,
catalecticant matrices and smoothable Gorenstein schemes.
{\it Journal of Algebraic Geometry}, to appear.
arXiv:1012.3563



\bibitem{ComSei11}
G.~Comas and M.~Seiguer.
On the rank of a binary form.
{\it Foundations of Computational Mathematics},
Vol.~11, No.~1, pp. 65-78, 2011.




\bibitem{Com00}
P.~Comon. Tensor decompositions - state of the art and applications.
Keynote address in IMA Conf. in signal processing, Warwick, UK,
2000.



\bibitem{CGLM08}
P. Comon, G. Golub, L.-H. Lim and B. Mourrain. Symmetric tensors
and symmetric tensor rank. {\em SIAM Journal on Matrix Analysis and
Applications}, 30, no. 3, 1254-1279, 2008.



%
%

\bibitem{ComMou96}
P. Comon and B. Mourrain.
Decomposition of quantics in sums of powers of linear forms.
{\it Signal Processing} \, 53(2), 93-107, 1996.
Special issue on high-order statistics.




\bibitem{CLO07}
D.~Cox, J.~Little and D.~O'Shea.
{\em Ideals, varieties, and algorithms: an introduction to
computational algebraic geometry and commutative algebra},
Springer, 2007.




\bibitem{CGT97}
R.~M.~Corless, P.~M.~Gianni and B.~M.~Trager.
A reordered Schur factorization method
for zero-dimensional polynomial systems with multiple roots.
{\it Proc. ACM Int. Symp.
Symbolic and Algebraic Computation}, 133-140, Maui, Hawaii, 1997.



\bibitem{CF96}
R. Curto and L. Fialkow.
Solution of the truncated complex moment problem for flat data.
{\it Memoirs of the American Mathematical Society}, 119(1996), No. 568,
Amer. Math. Soc., Providence, RI, 1996.



\bibitem{CF98}
R. Curto and L. Fialkow.
Flat extensions of positive moment matrices: recursively generated relations.
{\it Memoirs of the American Mathematical Society}, no. 648,
American Mathematical Society, Providence, 1998.


\bibitem{CF05}
R. Curto and L. Fialkow.
Truncated K-moment problems in several variables.
{\it Journal of Operator Theory},  54(2005), pp. 189-226.


\bibitem{DenSch83}
J. E. Dennis and R. B. Schnabel.
{\it Numerical methods for unconstrained optimization and nonlinear equations},
Prentice-Hall, Englewood Cliffs, NJ, 1983.



%
%


%
%
%


\bibitem{LLMRT13}
J.B.~Lasserre, M.~Laurent, B.~Mourrain, P.~Rostalski, and P.~Trebuchet.
Moment matrices, border bases and real radical computation.
{\it  J. Symbolic Computation}\, 51,  pp. 63--85.


%
%

%
%


\bibitem{Har}
J.~Harris.
{\it Algebraic geometry: a first course.}
Graduate Textbooks in Mathematics, Springer, 1992.


%
%


\bibitem{HiLi13}
C.~Hillar and L.-H.~Lim. Most tensor problems are NP-hard. {\em
Journal of the ACM}, 60 (2013), no. 6, Art. 45.

%
%


%
%



\bibitem{IarKan99}
A.~Iarrobino and V.~Kanev.
{\it Power Sums, Gorenstein algebras, and determinantal varieties.}
Lecture Notes in Mathematics \#1721,
Springer, 1999.



\bibitem{Kel95}
C. T. Kelley.
{\it Iterative methods for linear and nonlinear equations},
Frontiers in Applied Mathematics 16, SIAM, Philadelphia, 1995.


%
%



%
%

%
%


%
%
%
%



\bibitem{Land12}
J.M.~Landsberg. {\em Tensors: geometry and applications}. Graduate
Studies in Mathematics, 128. American Mathematical Society,
Providence, RI, 2012.


%
%
%
%



%
%



%
%


\bibitem{Lim13}
L.-H. Lim. Tensors and hypermatrices,
in: L. Hogben (Ed.),
{\it Handbook of Linear Algebra}, 2nd Ed., CRC Press, Boca Raton, FL, 2013.


%
%

\bibitem{More78}
J. J. More.
The Levenberg-Marquardt algorithm: implementation and theory,
in: G. A. Watson, ed.,
{\it Lecture Notes in Mathematics 630:
Numerical Analysis}, Springer-Verlag, Berlin, 1978,  105--116.



%
%


%
%
%
%
%
%
%
%
%


%
%

%
%
%
%


\bibitem{OedOtt13}
L.~Oeding and G.~Ottaviani. Eigenvectors of tensors and algorithms
for Waring decomposition. {\em J. Symbolic Comput.}\,  54, 9-35,
2013.

%
%


%
%

%
%


%
%

\bibitem{Sha:BAG1}
I.~Shafarevich.
{\it Basic algebraic geometry~1: varieties in projejctive space,}
second edition, Springer-Verlag, 1994.


%
%

\bibitem{Stu02}
B. Sturmfels. {\it Solving systems of polynomial equations}.
CBMS Regional Conference Series
in Mathematics, 97. American Mathematical Society, Providence, RI, 2002.


%
%
%
%
%
%
%
%
%
%


\bibitem{yyx11}
Y. X. Yuan.
Recent advances in numerical methods for nonlinear equations
and nonlinear least squares,
{\it Numerical Algebra Control and Optimization}, 1 (2011), 15--34.


%
%
%


%
%


%
%
%
%

\end{thebibliography}
\end{document}